\documentclass[a4paper,11pt, oneside]{amsart}
\usepackage{mathrsfs}
\usepackage{amsfonts}
\usepackage{amssymb}
\usepackage[english,polish]{babel}
\usepackage{amsxtra}
\usepackage{color}
\usepackage{dsfont}
\usepackage[margin=2cm, centering]{geometry}

\usepackage[colorlinks,citecolor=blue,urlcolor=blue,bookmarks=true]{hyperref}

\hypersetup{
pdfpagemode=UseNone,
pdfstartview=FitH,
pdfdisplaydoctitle=true,
pdfborder={0 0 0}, 
pdftitle={Asymptotic behaviour and estimates of slowly varying 
convolution semigroups},
pdfauthor={Tomasz Grzywny, Micha\l{} Ryznar and Bartosz Trojan},
pdflang=en-GB
}

\newcommand{\sprod}[2]{{\langle #1, #2\rangle}}
\newcommand{\norm}[1]{{\left\lvert #1 \right\rvert}}
\newcommand{\abs}[1]{{\lvert {#1} \rvert}}

\renewcommand{\atop}[2]{\substack{{#1}\\{#2}}}

\newcommand{\CC}{\mathbb{C}}

\newcommand{\NN}{\mathbb{N}}
\newcommand{\RR}{\mathbb{R}}
\newcommand{\PP}{\mathbb{P}}
\newcommand{\EE}{\mathbb{E}}

\newcommand{\calA}{\mathcal{A}}
\newcommand{\calR}{\mathcal{R}}
\newcommand{\calL}{\mathcal{L}}
\newcommand{\calS}{\mathcal{S}}

\newcommand{\calM}{\mathcal{M}}
\newcommand{\calB}{\mathcal{B}}

\newcommand{\calC}{\mathcal{C}}

\newcommand{\RInf}{\calR^\infty_{\alpha} }
\newcommand{\ROrg}{\calR^0_{\alpha} }

\newcommand{\pl}[1]{\foreignlanguage{polish}{#1}}

\newtheorem{theorem}{Theorem}[section]
\newtheorem{proposition}[theorem]{Proposition}
\newtheorem{lemma}[theorem]{Lemma}
\newtheorem{corollary}[theorem]{Corollary}
\newtheorem{claim}{Claim}

\numberwithin{equation}{section}

\theoremstyle{definition}
\newtheorem{example}{Example}
\newtheorem{remark}{Remark}

\newcommand{\ind}[1]{{\mathds{1}_{{#1}}}}

\newcommand{\Rd}{{\RR^{d}}}

\definecolor{mr}{rgb}{0.1,0.2,0.7}
\definecolor{tg}{rgb}{0.7,0.1,0.1}

\newcommand{\WUSC}[3]{\textrm{WUSC}({#1}, {#2}, {#3})}
\newcommand{\WLSC}[3]{\textrm{WLSC}({#1}, {#2}, {#3})}

\title[Asymptotic behaviour of densities]
{Asymptotic behaviour and estimates of slowly varying convolution semigroups}

\author{Tomasz Grzywny }
\address{
	\pl{
	Tomasz Grzywny\\
	Wydzia{\lll} Matematyki,
	Politechnika Wroc{\lll}awska\\
	Wyb. Wyspia\'{n}skiego 27\\
	50-370 Wroc{\lll}aw\\
	Poland}
}
\email{tomasz.grzywny@pwr.edu.pl}
\thanks{The first two authors were partially supported by the National Science Centre (Poland): grant 2015/17/B/ST1/01043.}

\author{Micha\l{} Ryznar}
\address{
	\pl{
	Micha\l{} Ryznar\\
	Wydzia\l{} Matematyki,
	Politechnika Wroc\l{}awska\\
	Wyb. Wyspia\'{n}skiego 27\\
	50-370 Wroc\l{}aw\\
	Poland}
}
\email{michal.ryznar@pwr.edu.pl}

\author{Bartosz Trojan}
\address{
	\pl{
	Bartosz Trojan\\
	Wydzia\l{} Matematyki,
	Politechnika Wroc\l{}awska\\
	Wyb. Wyspia\'{n}skiego 27\\
	50-370 Wroc\l{}aw\\
	Poland}
}
\email{bartosz.trojan@pwr.edu.pl}



\dedicatory{The authors dedicate this work to the memory of  Ante Mimica.}

\begin{document}
\selectlanguage{english}

\begin{abstract}
We prove the asymptotic formulas for the transition densities of isotropic unimodal convolution semigroups of
probability measures on $\mathbb{R} ^d$ under the assumption that its L\'{e}vy--Khintchine exponent varies slowly. We also derive some new estimates of the transition densities and Green functions.   
\end{abstract}

\maketitle


\section{Introduction}
Let $\mathbf{X} = (X_t : t \geq 0)$ be an isotropic unimodal L\'{e}vy process on $\RR^d$ with the
L\'{e}vy--Khintchine exponent $\psi$. In the recent paper \cite{MR3165234}, the estimates of the transition densities
$p(t,x)$ were studied under the assumption that $\psi$ has lower and upper Matuszewska indices strictly between $0$ and
$2$. In fact, more detailed information can be obtained whenever stronger assumption about the behaviour of $\psi$ is
imposed. Namely, in \cite{cgt}, it is proved that if $\psi$ varies regularly at infinity with index $\alpha \in (0, 2)$
then
\begin{equation}
	\label{RV}
	\lim_{\atop{x \to 0}{t \psi(\norm{x}^{-1}) \to 0}} \frac{p(t, x)}{t \norm{x}^{-d} \psi(\norm{x}^{-1})} = 
	\calA_{d, \alpha},
\end{equation}
and
\begin{equation}
	\label{eq:6}
	\lim_{\atop{t \to 0^+}{t\psi(\norm{x}^{-1}) \to \infty}}
	\frac{p(t, x)}{p(t, 0)}
	=
	1,
\end{equation}
where
\begin{align*}
	\mathcal{A}_{d, \alpha}
	=
	\alpha 2^{\alpha -1} 
	\pi^{-d/2-1}
	\sin \left(\frac{\alpha \, \pi}{2}\right)
	\Gamma \left(\frac{\alpha}{2}\right)
	\Gamma \left(\frac{\alpha+ d}{2}\right).
\end{align*}
Moreover, the asymptotics \eqref{RV} implies that $\psi$ varies regularly at infinity with index $\alpha \in (0, 2)$. A similar result was proved also for $\psi$ regularly varying at the origin.

The natural question arises: what are the asymptotic behaviour and estimates of the semigroup for the endpoints
$\alpha\in \{0, 2\}$. Regarding estimates, the case $\alpha= 2$ was investigated for subordinate Brownian motions in
the recent paper \cite{mimica}. However, the case $\alpha= 0$ seems to be hardly studied. 
In \cite[Theorem 5.52]{MR2569321} there are estimates of $p(t, x)$ for small values of time variable for a geometric
stable process, that is when
\[
	\psi(x) = \log (1 + \norm{x}^\beta)
\]
for some $\beta \in (0, 2)$. To the authors best knowledge the asymptotics were never studied before. This article is an
attempt to fill that gap in the theory. We provide a solution for a large class of slowly varying symbols including
geometric stable and iterated geometric stable cases. Namely, we study processes with symbols belonging to de
Haan class $\Pi_\ell^\infty$ associated with a function $\ell$ slowly varying at infinity, for definition
see \eqref{eq:5}.

For an isotropic unimodal L\'{e}vy process $\mathbf{X} = (X_t : t \geq 0)$ on $\RR^d$ with the L\'{e}vy--Khintchine
exponent $\psi \in \Pi^\infty_\ell$, which is equivalent to regular variation with index $-d$ at the origin of the density
of L\'evy's measure of the process (see Theorem \ref{thm:5}), we derive several asymptotic results describing the behaviour
of the transition density $p(t, x) $ and the Green function $G(x)$. To be more precise, in Section \ref{sec:3}, we prove
that
\begin{equation}
	\label{SV}
	\lim_{\atop{x \to 0}{t \psi(\norm{x}^{-1}) \to 0}}
	\frac{p(t, x)}{t \norm{x}^{-d} \ell(\norm{x}^{-1})}
	=
	\frac{\Gamma(d/2)}{2\pi^{d/2}}.
\end{equation}
We also study $\nu$ the density of the L\'evy measure of the process $\mathbf{X}$. In particular, \eqref{SV} implies
that
\[
	\lim_{\atop{x \to 0}{t \psi(\norm{x}^{-1}) \to 0}}
	\frac{p(t, x)}{t \nu(x)}
	=
	1.
\]
For the Green function $G$ of $\mathbf{X}$, if the process is transient (what always holds if $d\geq 3$),  we show that
\begin{align}\label{eq:35}
	\lim _{x \to 0} \frac{G(x)}{\norm{x}^{-d} \psi(\norm{x}^{-1})^{-2} \ell(\norm{x}^{-1})}
	=
	\frac{\Gamma(d/2)}{2 \pi^{d/2}}.
\end{align}
In particular, \eqref{eq:35} extends the result of \cite{MR2240700} where the geometric stable and iterated geometric
stable cases were studied. We also investigate the asymptotics of $p(t, x)$ when $t \psi(\norm{x}^{-1})$ gets large provided $\ell$ is a slowly 
varying bounded function. In Corollary \ref{cor:2} we show that
\begin{equation}
	\label{eq:1}
	\lim_{\atop{t \to 0^+}{t\psi(\norm{x}^{-1}) \to \infty}}
	\frac{p(t, x)}{t \norm{x}^{-d} \ell(\norm{x}^{-1}) e^{-t \psi(\norm{x}^{-1})}}
	=
	\frac{\Gamma(d/2)}{2 \pi^{d/2}}.
\end{equation}
We want to emphasize that if $t$ is sufficiently small then $p(t, 0)$ may be \emph{infinite}. This new phenomena requires
an additional factor not present in the case when $\psi$ is regularly varying at infinity with index strictly
larger than zero.

We also obtain the corresponding results when the L\'evy--Khintchine exponent $\psi$ belongs to $\Pi_\ell^0$ for some
function $\ell$ slowly varying at zero. Namely,
\[
	\lim_{\atop{\norm{x} \to \infty}{t \psi(\norm{x}^{-1}) \to 0}}
	\frac{p(t, x)}{t \norm{x}^{-d} \ell(\norm{x}^{-1})}
	=
	\frac{\Gamma(d/2)}{2\pi^{d/2}},
\]
and
\begin{equation*}
	\lim_{\norm{x} \to \infty} 
	\frac{G(x)}{\norm{x}^{-d} \psi(\norm{x}^{-1})^{-2} \ell(\norm{x}^{-1})}
	=
	\frac{\Gamma(d/2)}{2 \pi^{d/2}}.
\end{equation*}

Moreover, we prove several estimates of the transition density and the Green function under various assumptions
for unimodal and isotropic processes which improve the existing results for subordinate Brownian motions.

If the L\'{e}vy--Khintchine exponent $\psi$ belongs to de Haan class associated with a bounded function $\ell$ slowly
varying at infinity we show very precise two-sided estimates which complement the results obtained in \cite{MR3165234}.
Namely, we show that there are positive $r_0$ and $t_0$ such that for all $\norm{x} \leq r_0$ and $t \in (0, t_0)$
\footnote{$A(x) \asymp B(x)$ for $x \in I$ means that there $C \geq 1$ such that
$C^{-1} B(x) \leq A(x) \leq C B(x)$ for all $x \in I$}
\begin{equation*}
	p(t, x)
	\asymp
	t \norm{x}^{-d} \ell(\norm{x}^{-1}) e^{-t \psi(\norm{x}^{-1})}.
\end{equation*}
To our best knowledge this is the first result of this type when the L\'{e}vy--Khintchine exponent is slowly varying at
infinity

There are many articles exhibiting the local behaviour of the transition density of L\'evy processes for small time
and space variables which are valid under certain assumptions on the density of the L\'evy measure, see e.g.
\cite{MR2492992, MR898496, MR2806700, MR2357678, MR3357585, MR2886383}. Namely, if $\nu(x) |x|^{d}$ is comparable at
the origin with a decreasing $V(|x|)$ having the Matuszewska index at the origin strictly between $0$ and $2$ then 
there are positive $r_0$ and $t_0$ such that for all $\norm{x} \leq r_0$ and $t \in (0, t_0)$
\begin{equation*}	
	p(t, x) \asymp \min\big\{t \norm{x}^{-d}V(|x|), p(t,0)\big\}.
\end{equation*}
One of our contributions is to discover that for processes with the L\'{e}vy--Khintchine exponent belonging to de Haan
class associated with a bounded function slowly varying at infinity the r\^ole of $p(t,0)$ is taken by
$t \norm{x}^{-d} \ell(\norm{x}^{-1}) e^{-t \psi(\norm{x}^{-1})}$. To illustrate the results we provide examples
with slowly varying symbols to which our method applies. These examples suggest that a unified form of estimates
may be difficult to discover within the class of processes with slowly varying symbols since the estimates exhibit
quite a lot of irregularity.   
	
In the case of subordinate Brownian motions we prove some extensions of the results of \cite{MR3165234} where
estimates of $p(t,x)$ were obtained in terms of the symbol $\psi$ of the process under weak scaling assumptions.
Namely, $\psi$ satisfies lower and upper scaling conditions at infinity
with indices strictly between $0$ and $2$ if and only if there is $r_0$ such that for all $t > 0$ and $\norm{x} \leq r_0$,
if $t \psi(\norm{x}^{-1}) \leq 1$ then
\begin{equation}
	\label{BGR}
	p(t, x)\asymp \frac{t\psi(|x|^{-1})}{|x|^{d}}.
\end{equation}
Let us recall that for a subordinate Brownian motion we have $\psi(x)= \varphi(|x|^2)$ where $\varphi$ is a Bernstein
function which determines the distribution of the underlying subordinator via Laplace transform. It turns out that
suitable weak scaling properties imposed on the derivative  $\varphi^\prime$ lead to the following extension of
\eqref{BGR}, there is $r_0 > 0$ such that for all $t > 0$ and $\norm{x} < r_0$, if $t \varphi(\norm{x}^{-2}) \leq 1$ then
\begin{equation}
	\label{eq:37}
	p(t, x)
	\asymp 
	\frac{t\varphi'(|x|^{-2})}{|x|^{d+2}}.
\end{equation}
The last result is particularly interesting when $\psi$ is slowly varying at infinity since \eqref{BGR} can not hold
in this case. Moreover, under some additional assumptions on the subordinator we can also show the converse. That is,
if \eqref{eq:37} is satisfied then $\varphi^\prime$ has some scaling properties.

For subordinate Brownian motion, the case when $\psi$ is regularly varying at $\infty$ with index $2$, a function
$H(r)= \varphi(r)- r\varphi'(r)$ plays a similar r\^ole as $\varphi'$ in this article. For details we refer to
the recent paper \cite{mimica} where the estimates were obtained under suitable scaling assumptions of $H$.

We can conclude that our present paper together with \cite{mimica} shed some light on the situation when the
L\'{e}vy--Khintchine exponent $\psi$ is slowly varying or $2$-regularly varying at infinity. These cases seem to be more
complicated and still far from being completely understood.
	
The paper is organized as follows. In Section \ref{sec:2}, we recall some basic facts regarding unimodal isotropic L\'{e}vy
processes. In Section \ref{sec:3}, we present an extension of uniform Tauberian theorems we apply to derive asymptotics of $p(t,x)$
under the assumption that the L\'{e}vy--Khintchine exponent is slowly varying and belongs to de Haan class. We also study
the asymptotics of the Green function at the origin and at infinity. In Section \ref{sec:5}, we derive general upper and lower bounds
for $p(t,x)$ and the Green function under various assumptions, usually involving scaling properties of some functions
determined by $\psi$. Next, for subordinate Brownian motion we show two-sided estimates of the resolvent kernel
$G^\lambda$ under assumptions on scaling properties of the derivative of the Laplace exponent of the subordinator.
Then we use the obtained estimates to derive estimates of $p(t,x)$. Finally, in Section \ref{sec:6} we give examples
illustrating the behaviour of $p(t,x)$ for processes with slowly varying L\'{e}vy--Khintchine exponents.

\section{Preliminaries}
\label{sec:2}
Let $\mathbf{X}=(X_t : t \geq 0)$ be an isotropic pure jump L\'{e}vy process in $\RR ^d$, i.e. $\mathbf{X}$ is
a c\'{a}dl\'{a}g stochastic process with a distribution denoted by $\PP$ such that $X_0=0$ almost surely, the increments
of $\mathbf{X}$ are independent with a radial distribution $p(t, \,\cdot\,)$ on $\RR ^d\setminus \{0\}$. This is
equivalent with radiality of the L\'{e}vy measure $\nu$ and the L\'{e}vy--Khintchine exponent $\psi$. In particular,
the characteristic function of $X_t$ has a form
\begin{align}
	\label{eq:29}
	\int_{\RR^d} e^{i \sprod{\xi}{x}} p(t, {\rm d} x) = e^{-t\psi (\xi)}
\end{align} 
where
\begin{equation}
	\label{eq:30}
	\psi (\xi) = \int_{\RR ^d} \big(1-\cos\sprod{\xi}{x}\big) \: \nu ({\rm d}x).
\end{equation}
We are going to abuse the notation by setting $\psi(r)$ for $r \ge 0$ to be equal to $\psi(\xi)$ for any $\xi \in \RR^d$
with $\norm{\xi} = r$. The same rule applies to any radial function appearing in this paper. Since the function $\psi$
is not necessarily radially monotonic, it is conveniently to work with $\psi^*$ defined for $u \geq 0$ by 
\begin{equation*}
	\psi^*(u) = \sup_{s \in [0, u]} \psi(s).
\end{equation*}
Let us recall that for $r, u \geq 0$ (see \cite[Theorem 2.7]{WHoh})
\begin{equation}
	\label{eq:31}
	\psi (ru)
	\leq 
	\psi^*(ru)
	\leq 2(r^2+1)
	\psi^*(u).
\end{equation}

A Borel measure $\mu$ is isotropic \emph{unimodal} if it is absolutely continuous on $\RR^d \setminus \{0\}$
with a radial and radially non-increasing density. A L\'{e}vy process $\mathbf X$ is isotropic unimodal if
$p(t, \: \cdot \:)$ is isotropic unimodal for each $t > 0$. We consider a subclass of isotropic processes
consisting of isotropic unimodal L\'{e}vy processes. They were characterized by Watanabe in \cite{Watanabe} as those
having the isotropic unimodal L\'{e}vy measure. Hence, the L\'evy measure is absolutely continuous with respect to
the Lebesgue measure. Its density, denoted by $\nu(x)=\nu(\norm{x})$, is radially non-increasing.
A remarkable property of these processes is (see \cite[Proposition 2]{MR3165234})
\begin{equation}
	\label{eq:32}
	\psi^*(u) \leq \pi^2 \psi(u)
\end{equation}
for all $u \geq 0$.

In this article,  we mainly consider isotropic unimodal L\'{e}vy processes with the L\'{e}vy--Khintchine exponent $\psi$
slowly varying at infinity or zero. Let us recall that a function $\ell: [a, \infty) \rightarrow [0, \infty)$,
for some $a > 0$, is called \emph{slowly varying at infinity} if for each $\lambda > 0$
\[	
	\lim_{x \to \infty} \frac{\ell(\lambda x)}{\ell(x)} = 1.
\]
The set of functions slowly varying at infinity is denoted by $\calR^\infty_0$. A positive function $\ell$ defined
in a right neighbourhood of zero is slowly varying at zero if $\ell(x^{-1})$ is in $\calR^\infty_0$. The set of functions
slowly varying at zero is denoted by $\calR^0_0$.

The following property of a function slowly varying at infinity appears to be very useful (see \cite{pot}, see also
\cite[Theorem 1.5.6]{bgt}). For every $C > 1$ and $\epsilon > 0$ there is $x_0 \geq a$ such that for all
$x, y \geq x_0$
\begin{equation}
	\label{eq:14}
	\ell(x) \leq C \ell(y) \max\{x/y, y/x\}^\epsilon.
\end{equation}
Given a function $\ell$ slowly varying at infinity, by $\Pi^\infty_\ell$ we denote a class of functions
$f: [a, \infty) \rightarrow [0, \infty)$ such that for all $\lambda > 0$
\begin{equation}
	\label{eq:5}
	\lim_{x \to \infty} \frac{f(\lambda x) - f(x)}{\ell(x)} = \log \lambda.
\end{equation}
The collection  $\Pi^\infty_\ell$ is called \emph{de Haan class at infinity} determined by $\ell$.  Similarly we define de Haan class at the origin for $\ell\in \calR^0_0$.  A positive function $f$ defined
in a right neighbourhood of zero belongs to  $\Pi^0_\ell$  if \begin{equation}
	\label{eq:5a}
	\lim_{x \to 0^+} \frac{f(\lambda x) - f(x)}{\ell(x)} = \log \lambda,\quad \lambda>0.
\end{equation}

In Sections \ref{sec:4} and \ref{sec:5}, we consider a class of functions larger than $\calR^\infty_0$. Namely, we say
that $f: [0, \infty) \rightarrow [0, \infty)$ satisfies weak lower scaling condition at infinity if there are
$\alpha \in \RR$, $x_0 \geq 0$ and $c \in (0, 1]$ such that for all $\lambda > 1$ and $x > x_0$
\[
	f(\lambda x) \geq c \lambda^\alpha f(x).
\]
Then we write $f \in \WLSC{\alpha}{x_0}{c}$. Similarly, $f$ satisfies the weak upper scaling condition at infinity
if there are $\beta \in \RR$, $x_0 \geq 0$ and $C \in [1, \infty)$ such that for all $\lambda \geq 1$ and
$x > x_0$
\[
	f(\lambda x) \leq C \lambda^\beta f(x).
\]
Then $f \in \WUSC{\beta}{x_0}{C}$. 

Finally, for a function $f: [0, \infty) \rightarrow \RR$ its Laplace transform is defined by
\[
	\calL f(\lambda) = \int_0^{\infty} e^{-\lambda x} f(x) {\: \rm d}x,
\]
whereas the Laplace--Stieltjes transform
\[
	\calL\{{\rm d} f\}(\lambda) = \int_{[0, \infty)} e^{-\lambda x} {\: \rm d} f(x).
\]
We have $\lambda \calL f (\lambda) = \calL \{{\rm d} f\}(\lambda)$.

\section{Asymptotics for heat kernels and Green functions}
\label{sec:3}
In this section we consider an isotropic unimodal L\'{e}vy process $\mathbf X = \big(X_t : t \geq 0\big)$ on $\RR^d$
having the L\'{e}vy--Khintchine exponent $\psi$ in $\Pi_\ell^\infty$ for some function $\ell$ slowly varying at
infinity. We describe an asymptotic behaviour of the transition density $p(t, x)$ close to
the origin when $t \psi\big(\norm{x}^{-1}\big)$ is small (see Theorem \ref{thm:2}). In the case when $\ell$ is
bounded we also give the asymptotics at the origin as $t \psi\big(\norm{x}^{-1}\big)$ gets large (see Theorem \ref{thm:3}).

\subsection{Tauberian theorem}
We start by proving a slightly more general Tauberian theorem.

\begin{theorem}
	\label{thm:1}
	Let $\{Q_t : t \geq 0\}$ be a family of non-decreasing and non-negative functions on $[0, \infty)$
	such that there are two families of positive functions $\{q_t : t \geq 0\}$ and $\{w_t : t \geq 0\}$
	satisfying
	\begin{equation}
		\label{eq:2}
		\lim_{\atop{\lambda \to \infty}{w_t(\lambda) \to 0}} 
		\frac{\lambda \calL\{{\rm d} Q_t\}(\lambda)}{q_t(\lambda)} = 1.
	\end{equation}
	We assume that
	\begin{enumerate}
		\item there are $C, a > 0$ such that for all $x > 0$ and $t \geq 0$
		\[
			\frac{Q_t(x)}{x} \leq 
			\begin{cases}
				C q_t(x^{-1}) & \text{if}\quad 0 < x \leq a^{-1},\\
				C q_t(a) & \text{otherwise};
			\end{cases}
		\]
		\item there are $C, a > 0$ and $\eta \in (0, 1)$ such that for all $x, \lambda > 0$ and $t \geq 0$ if
		$x \geq a$ and $\lambda x \geq a$ then
		\[
			q_t(\lambda x) \leq C q_t(x) \max\{\lambda, \lambda^{-1}\}^{\eta};
		\]
		\item there is $\rho \geq 0$ such that for all $\lambda > 0$
		\[
		\lim_{\atop{x \to \infty}{w_t(x) \to 0}} \frac{q_t(\lambda x)}{q_t(x)} = \lambda^\rho.
		\]
	\end{enumerate}
	Then
	\begin{equation}
		\label{eq:3}
		\lim_{\atop{r \to 0}{w_t(r^{-1}) \to 0}} \frac{Q_t(r)}{r q_t(r^{-1})} = \frac{1}{\Gamma(\rho+2)}.
	\end{equation}
\end{theorem}
\begin{proof}
	For any pair $(t, r) \in [0, \infty) \times (0, a^{-1})$ we define a tempered distribution $\Lambda_{t, r}$
	by setting for $f \in \calS\big([0, \infty)\big)$
	\[
		\Lambda_{t, r}(f)
		=
		\frac{1}{r q_t(r^{-1})} \int_0^{\infty}
		f(x)
		Q_t(x r)
		{\: \rm d}x.
	\]
	Let us recall that the space $\calS\big([0, \infty)\big)$ consists of Schwartz functions on $\RR$ restricted to
	$[0, \infty)$ and $\calS^\prime \big([0, \infty)\big)$ consists of tempered distributions which are supported
	by $[0,\infty )$ (see \cite{vlad} for details).
	
	The upper bounds on $Q_t$ imply
	\begin{align*}
		\Lambda_{t,r}(f)
		& \leq 
		C \int_0^{\frac{1}{a r}} \frac{q_t(x^{-1} r^{-1})}{q_t(r^{-1})} x \abs{f(x)} {\: \rm d}x
		+ C \int_{\frac{1}{a r}}^{\infty} \frac{q_t(a)}{q_t(r^{-1})} x \abs{f(x)} {\: \rm d}x.
	\end{align*}
	By (ii) we have
	\[
		\int_0^{\frac{1}{a r}} \frac{q_t(x^{-1} r^{-1})}{q_t(r^{-1})} x \abs{f(x)} {\: \rm d}x
		\leq
		C
		\int_0^{\frac{1}{a r}} \max\{x^{1-\eta}, x^{1+\eta}\} \abs{f(x)} {\: \rm d}x,
	\]
	and
	\begin{align*}
		\int_{\frac{1}{a r}}^{\infty} \frac{q_t(a)}{q_t(r^{-1})} x \abs{f(x)} {\: \rm d}x
		& \leq
		C
		\int_{\frac{1}{a r}}^{\infty} (a r)^{-\eta} x \abs{f(x)} {\: \rm d}x \\
		& \leq	
		C
		\int_{\frac{1}{a r}}^{\infty} x^{1+\eta} \abs{f(x)} {\: \rm d}x.
	\end{align*}
	Hence, we may estimate
	\[
		\Lambda_{t, r}(f) \leq C \sup_{x \in [0, \infty)} (1+x^2) \abs{f(x)},
	\]
	i.e. the family $\big\{\Lambda _{t,r}: (t,r) \in [0, \infty) \times (0, a^{-1})\big\}$ is equicontinuous.

	Next, for $\tau > 0$, we consider a function $f_\tau(x) = e^{-\tau x}$. We have
	\begin{align*}
		\Lambda_{t, r} (f_\tau) 
		& = \frac{1}{r q_t(r^{-1})} \int_0^{\infty} e^{-\tau x} Q_t(x r) {\: \rm d}x \\
		& = \frac{1}{r^2 q_t(r^{-1})} \calL Q_t(\tau r^{-1}) \\
		& = \frac{1}{r \tau q_t(r^{-1})} \calL\{{\rm d} Q_t\}(\tau r^{-1}) \\
		& = 
		\frac{1}{\tau^2}
		\cdot
		\frac{q_t(r^{-1} \tau^{-1})}{q_t(r^{-1})} 
		\cdot
		\frac{\tau r^{-1} \calL\{{\rm d} Q_t\} (\tau r^{-1})}{q_t(r^{-1} \tau^{-1})}.
	\end{align*}
	Therefore, by \eqref{eq:2} and (iii), we obtain
	\begin{equation}
		\label{eq:48}
		\lim_{\atop{r \to 0}{w_t(r^{-1}) \to 0}} 
		\Lambda_{t, r}(f_\tau)
		=
		\tau^{-\rho - 2}
		=
		\frac{1}{\Gamma(\rho+2)}
		\int_0^{\infty} e^{-\tau x} x^{\rho+1}{\: \rm d}x.
	\end{equation}
	Let $\calB$ be the linear span of the set $\{f_\tau : \tau >0\}$.  Since $\calB$ is dense in
	$\calS\big([0, \infty)\big)$ and the family $\big\{\Lambda _{t,r}: (t,r) \in [0, \infty) \times (0, a^{-1})\big\}$
	is equicontinuous on $\calS\big([0, \infty)\big)$, from \eqref{eq:48} we obtain that for any
	$f \in \calS\big([0, \infty)\big)$,
	\[
		\lim_{\atop{r \to 0}{w_t(r^{-1}) \to 0}}
		\Lambda_{t, r}(f)
		=
		\frac{1}{\Gamma(\rho+2)} \int_0^{\infty} f(x) x^{\rho+1} {\: \rm d}x.
	\]
	To conclude the proof of theorem we consider a specific function $f$. For a given $\epsilon > 0$, let
	$\phi_+ \in \calS\big([0, \infty)\big)$ be such that $0 \leq \phi_+ \leq 1$ and 
	\[
		\phi_+(x) =
		\begin{cases}
			1 & \text{ for } 0 \leq x \leq 1, \\
			0 & \text{ for } 1+\varepsilon \leq x.
		\end{cases}
	\]
	Then
	\begin{align*}
		\frac{Q_t(r)}{r q_t\big(r^{-1}\big)}
		&\leq 
		\frac{1}{r q_t\big(r^{-1}\big)} \int_{[0, r]}  \phi_+(s/r) {\: \rm d} Q_t(s) \\
		& \leq
		\frac{1}{r q_t\big(r^{-1}\big)} \int_{[0, \infty]}  \phi_+(s/r) {\: \rm d} Q_t(s),
	\end{align*}
	thus
	\[
	\frac{Q_t(r)}{r q_t \big(r^{-1}\big)}  \leq - \Lambda_{t, r}(\phi_+').
	\]
	Hence,
	\begin{align*}
		\limsup_{\atop{r \to 0}{w_t(r^{-1}) \to 0}}
		\frac{Q_t(r)}{r q_t \big(r^{-1}\big)}
		&\leq 
		-\frac{1}{\Gamma(\rho+2)} \int_0^{\infty} \phi_+'(x) x^{\rho+1} {\: \rm d}x \\
		& =
		\frac{1}{\Gamma(\rho+1)} \int_0^{\infty} \phi_+(x) x^{\rho} {\: \rm d}x \\
		& \leq
		\frac{(1+\varepsilon)^{\rho+1}}{\Gamma(\rho+2)}.
	\end{align*}
	Similarly, taking $\phi_- \in \calS\big([0, \infty)\big)$ such that $0 \leq \phi_- \leq 1$ and
	\[
		\phi_-(x) =
		\begin{cases}
			1 & \text{ for } 0 \leq x \leq 1-\varepsilon, \\
			0 & \text{ for }1 \leq x,
		\end{cases}
	\]
	we can show that
	\[
		\liminf_{\atop{r \to 0}{w_t(r^{-1}) \to 0}} 
		\frac{Q_t(r)}{r q_t\big(r^{-1}\big)}
		\geq \frac{(1 - \varepsilon)^{\rho+1}}{\Gamma(\rho+2)},
	\]
	which ends the proof.
\end{proof}

Essentially the same proof gives the following theorem.
\begin{theorem}
	\label{thm:8}
	Let $\{Q_t : t \geq 0\}$ be a family of non-decreasing and non-negative functions on $[0, \infty)$
	such that there are two families of positive functions $\{q_t : t \geq 0\}$ and $\{w_t : t \geq 0\}$
	satisfying
	\[
		\lim_{\atop{t \to 0^+}{w_t(\lambda) \to \infty}} 
		\frac{\lambda \calL\{{\rm d} Q_t\}(\lambda)}{q_t(\lambda)} = 1.
	\]
	We assume that
	\begin{enumerate}
		\item there are $C, a, \epsilon > 0$ such that for all $x > 0$ and $t \in (0, \epsilon)$
		\[
			\frac{Q_t(x)}{x} \leq 
			\begin{cases}
				C q_t(x^{-1}) & \text{if}\quad 0 < x \leq a^{-1},\\
				C q_t(a) & \text{otherwise};
			\end{cases}
		\]
		\item there are $C, a, \epsilon > 0$ and $\delta \in (0, 1)$ such that for all $x, \lambda > 0$ if
		$x \geq a$, $\lambda x \geq a$ and $t \in (0, \epsilon)$ then
		\[
			q_t(\lambda x) \leq C q_t(x) \max\{\lambda, \lambda^{-1}\}^{\delta};
		\]
		\item there is $\rho \geq 0$ such that for all $\lambda > 0$
		\[
		\lim_{\atop{t \to 0^+}{w_t(x) \to \infty}} \frac{q_t(\lambda x)}{q_t(x)} = \lambda^\rho.
		\]
	\end{enumerate}
	Then
	\[
		\lim_{\atop{t \to 0^+}{w_t(r^{-1}) \to \infty}} \frac{Q_t(r)}{r q_t(r^{-1})} = \frac{1}{\Gamma(\rho+2)}.
	\]
\end{theorem}

\subsection{Small $t \psi\big(\norm{x}^{-1}\big)$}
To get the asymptotics of $p(t, x)$ as $t \psi\big(\norm{x}^{-1}\big)$ approaches zero we are going to apply Theorem
\ref{thm:1} to
\begin{equation}
	\label{eq:16}
	Q_t(r) = \int_0^{\sqrt{r}} u^{d+1} p(t, u) {\: \rm d}u,
\end{equation}
with $q_t(r) = t \ell\big(\sqrt{r}\big)$ and $w_t(r) = t \psi\big(\sqrt{r}\big)$.
\begin{proposition}
	\label{prop:1}
	If $\psi \in \Pi_\ell^\infty$ then there are $C, \lambda_0 > 0$ such that
	\begin{equation}
		\label{eq:10}
		\frac{Q_t(\lambda)}{\lambda} \leq 
		\begin{cases}
		C t \ell(\lambda^{-1/2}) & \text{if} \quad 0 < \lambda \leq \lambda_0, \\
		C t \ell(\lambda_0^{-1/2}) & \text{otherwise.}
		\end{cases}
	\end{equation}
	Moreover,
	\begin{equation}
		\label{eq:9}
		\lim_{\atop{\lambda \to \infty}{t \psi(\sqrt{\lambda}) \to 0}} 
		\frac{\lambda \calL\{{\rm d} Q_t\}(\lambda)}{t \ell(\sqrt{\lambda})}
		= \frac{1}{2}.
	\end{equation}
\end{proposition}
\begin{proof}
	Let $a \geq 1$. First, we prove the following claim.
	\begin{claim}
		\label{clm:3}
		There are $C, \delta, x_0 > 0$ such that for all $r, \lambda > 0$ and $t \geq 0$, if
		$\sqrt{\lambda} \geq x_0$ then
		\begin{equation}
			\label{eq:7}
			\big|
			e^{-t \psi(r \sqrt{\lambda})} - e^{-t \psi(r \sqrt{\lambda a})}
			\big|
			\leq
			C t \ell(\sqrt{\lambda}) \max\{r, r^{-1}\}^\delta.
		\end{equation}
	\end{claim}
	Indeed, by \eqref{eq:14} and \eqref{eq:5}, there are $C > 0$ and $x_0 \geq 0$ such that for all $r, \lambda > 0$, if
	$r \sqrt{\lambda}, \sqrt{\lambda} \geq x_0$ then
	\[
		\big|\psi(r \sqrt{\lambda}) - \psi(r \sqrt{\lambda a})\big|
		\leq
		C \ell(r \sqrt{\lambda}),
	\]
	and
	\[
		\ell(r \sqrt{\lambda}) \leq 2 \ell(\sqrt{\lambda}) \max\{r, r^{-1}\}^\delta.
	\]
	Let $\sqrt{\lambda} \geq x_0$. If $r \sqrt{\lambda} \geq x_0$ then
	\begin{align*}
		\big| e^{-t \psi(r \sqrt{\lambda})} - e^{-t \psi(r \sqrt{\lambda a})} \big|
		& \leq
		t \big| \psi(r \sqrt{\lambda} ) - \psi(r \sqrt{\lambda a}) \big| \\
		& \leq C t \ell(\sqrt{\lambda}) \max\{r , r^{-1}\}^\delta.
	\end{align*}
	Otherwise, $r \sqrt{\lambda} \leq x_0$ and we can estimate
	\begin{align*}
		\big| e^{-t \psi(r \sqrt{\lambda})} - e^{-t \psi(r \sqrt{\lambda a})} \big|
		& \leq
		t \big| \psi(r \sqrt{\lambda} ) - \psi(r \sqrt{\lambda a}) \big| \\
		& \leq
		2 t \psi^*(r \sqrt{\lambda a}) \\
		& \leq
		C t \ell(x_0).
	\end{align*}
	Since, by \eqref{eq:14}, we have
	\[
		\ell(x_0) \leq C \ell(\sqrt{\lambda}) r^{-\delta},
	\]
	the estimate \eqref{eq:7} follows.
	
	Next, let us consider
	\begin{equation}
		\label{eq:34}
		U_t(r)
		=\PP\big(0<\abs{a_t} \leq \sqrt{r}\big) \\
		= \frac{2 \pi^{d/2}}{\Gamma(d/2)} \int^{\sqrt{r}}_{0} u^{d-1} p(t, u) {\: \rm d} u.
	\end{equation}
	We observe that by the Fubini--Tonelli's theorem
	\[
		\calL\{{\rm d} U_t\} (\lambda) = \lambda \calL U_t(\lambda) 
		= \int_{\RR^d} e^{-\lambda \norm{x}^2} p(t, x) {\: \rm d}x.
	\]
	Since
	\begin{equation}
		\label{eq:33}
		e^{-\lambda \norm{x}^2} 
		= (4\pi)^{-d/2} 
		\int_{\RR^d} e^{-\norm{\xi}^2/4} e^{-i \sqrt{\lambda} \sprod{\xi}{x}} {\: \rm d}\xi,
	\end{equation}
	by the second application of the Fubini--Tonelli's theorem we get
	\begin{align}
		\nonumber
		\calL\{{\rm d} U_t\} (\lambda)
		&= (4 \pi)^{-d/2} \int_{\RR^d} e^{-t\psi(\xi \sqrt{\lambda})}
		e^{-\norm{\xi}^2/4} {\: \rm d}\xi - \PP\big(\abs{X_t}=0  \big)\nonumber\\
		\label{eq:26}
		&= \frac{2^{1-d}}{\Gamma(d/2)} \int_0^{\infty} e^{-t\psi(r \sqrt{\lambda})} 
		e^{-r^2/4} r^{d-1} {\: \rm d}r -\PP\big(\abs{X_t}=0  \big)
	\end{align}
	where in the last step we used polar coordinates.

	We now claim that
	\begin{claim}
		\label{clm:1}
		For $a\geq 1$
		\begin{align}
			\label{eq:4}
			\lim_{\atop{\lambda \to \infty}{t \psi(\sqrt{\lambda}) \to 0}}
			\frac{\calL\{{\rm d} U_t\}(a \lambda)-\calL\{{\rm d} U_t\}(\lambda)}{ t\ell (\sqrt{\lambda })}
			=
			\frac{\log a}{2}.
	\end{align}
	\end{claim}
	For the proof of the claim, by \eqref{eq:26}, we write
	\begin{equation}
		\label{eq:12}
		\frac{\calL\{{\rm d}U_t\}(a \lambda) - \calL\{{\rm d}U_t\}(\lambda)}{ t\ell (\sqrt{\lambda })}
		=
		\frac{2^{1-d}}{\Gamma(d/2)} \int_0^{\infty}
		\frac{e^{-t\psi(r\sqrt{\lambda a})} - e^{-t\psi(r \sqrt{\lambda})}}{ t\ell (\sqrt{\lambda })}
		e^{-r^2/4} r^{d-1} {\: \rm d}r.
	\end{equation}
	By \eqref{eq:7}, the integrand in \eqref{eq:12} is uniformly bounded by an integrable function.
	Since for a fixed $r > 0$, 
	\[
		\frac{e^{-t\psi(r\sqrt{\lambda a})}-e^{-t\psi(r\sqrt{\lambda})}}{ t\ell (\sqrt{\lambda})}
		= 
		\frac{\ell (r\sqrt{\lambda})}{\ell (\sqrt{\lambda})}
		\cdot
		\frac{\psi(r\sqrt{\lambda a})-\psi(r\sqrt{\lambda})}{\ell (r\sqrt{\lambda})}
		\cdot
		\frac{e^{-t\psi(r\sqrt{\lambda a})}-e^{-t\psi(r\sqrt{\lambda})}}
		{t\psi(r\sqrt{\lambda a})-t\psi(r\sqrt{\lambda})},
	\]
	we obtain
	\[
		\lim_{\atop{\lambda \to \infty}{t \psi(\sqrt{\lambda}) \to 0}}
		\frac{e^{-t\psi(r\sqrt{\lambda a})}-e^{-t\psi(r\sqrt{\lambda})}}{ t\ell (\sqrt{\lambda})}
		= \frac{\log a}{2},
	\]
	hence, by the dominated convergence theorem we obtain \eqref{eq:4}.

	Notice, that by \eqref{eq:7} we also get
	\begin{equation}
		\label{eq:8}
		\calL\{{\rm d}U_t\}(\lambda a) - \calL\{{\rm d}U_t\}(\lambda)
		\leq C t \ell(\sqrt{\lambda}),
	\end{equation}
	provided that $\sqrt{\lambda} \geq x_0$ and $t \geq 0$. 

	We now turn to the proof of \eqref{eq:9}. Observe that $\calL\{{\rm d}Q_t\} = \big(\calL\{{\rm d} U_t\}\big)'$.
	Hence,
	\[
		\calL \{{\rm d} U_t\}(s^{-1}) = \int_0^s \calL \{{\rm d} Q_t\}(r^{-1}) r^{-2} {\: \rm d}r.
	\]
	By monotonicity of the Laplace--Stieltjes transform of $U_t$ we infer that for any $a > 1$
	\begin{equation}
		\label{eq:11}
		\begin{aligned}
		s^{-1} \calL\{{\rm d} Q_t\}(s^{-1})
		& \geq
		\frac{1}{a-1} 
		\left(\calL\{{\rm d} U_t\}(s^{-1})-\calL\{{\rm d} U_t\} (as^{-1})\right), \\
		s^{-1} \calL\{{\rm d} Q_t\} (s^{-1})
		& \leq
		\frac{a}{a-1}\left(\calL\{{\rm d} U_t\}((as)^{-1})-\calL\{{\rm d}U_t\}(s^{-1})\right).
		\end{aligned}
	\end{equation}
	Therefore,
	\begin{equation*}
		\begin{aligned}
		\lim_{\atop{\lambda \to \infty}{t \psi(\sqrt{\lambda}) \to 0}}
		\frac{\lambda \calL \{{\rm d} Q_t\}(\lambda)} {t \ell(\sqrt{\lambda})}
		\geq \frac{1}{2} \frac{\log a}{a-1}, \\
		\lim_{\atop{\lambda \to \infty}{t \psi(\sqrt{\lambda}) \to 0}}
		\frac{\lambda \calL \{{\rm d} Q_t\}(\lambda)} {t \ell(\sqrt{\lambda})}
		\leq \frac{1}{2} \frac{a \log a}{a-1},
		\end{aligned}
	\end{equation*}
	thus, by taking $a$ tending to $1$ we get \eqref{eq:9}. 

	Finally, since $Q_t$ is non-decreasing, by \eqref{eq:11},
	\begin{align}
		\nonumber
		Q_t(x) 
		& \leq e \calL\{{\rm d} Q_t\}(x^{-1}) \\
		\label{eq:23}
		& \leq 2 e x \left(\calL\{{\rm d} U_t\}(2^{-1} x^{-1}) - \calL\{{\rm d} U_t\}(x^{-1})\right).
	\end{align}
	Therefore, if $2 x \leq x_0^{-2}$, by \eqref{eq:8},
	\[
		Q_t(x) 
		\leq C x t \ell(x^{-1/2}).
	\]
	For $2 x > x_0^{-2}$, by \eqref{eq:31} we estimate
	\[
		\Big\lvert
		e^{-t \psi(r x^{-1/2})} - e^{-t\psi(r x^{-1/2} 2^{-1/2})} 
		\Big\rvert
		\leq
		2 t \psi^*(r x^{-1/2})
		\leq
		2 t (1+r^2) \psi^*( x_0).
	\]
	Hence, by \eqref{eq:12} and \eqref{eq:23} we get
	\begin{align*}
		Q_t(x) \leq C t x \ell(x_0),
	\end{align*}
	which finishes the proof. 
\end{proof}

\begin{theorem}
	\label{thm:2}
	Let $\mathbf{X}$ be an isotropic unimodal L\'{e}vy process on $\RR^d$ with the L\'{e}vy--Khintchine
    exponent $\psi \in \Pi^\infty_\ell$ for some $\ell \in \calR^\infty_0$. Then
	\[
		\lim_{\atop{x \to 0}{t \psi(\norm{x}^{-1}) \to 0}}
		\frac{p(t, x)}{t \norm{x}^{-d} \ell(\norm{x}^{-1})}
		=
		\frac{\Gamma(d/2)}{2\pi^{d/2}}.
	\]
\end{theorem}
\begin{proof}
	In the proof we use an argument from \cite[Theorem 1.7.2]{bgt}. For any $0<a<b$, we have
	\begin{align*}
		Q_t(b r)-Q_t(a r)
		=
		\frac{2 \pi^{d/2}}{\Gamma(d/2)} 
		\int_{\sqrt{ar}}^{\sqrt{br}} u^{d+1} p(t, u) {\: \rm d} u.
	\end{align*} 
	Since the function $u\mapsto p(t, u)$ is non-increasing, we get
	\begin{equation}
		\label{eq:15}
		\begin{aligned}
		\frac{Q_t(br)-Q_t(ar)}{t\ell(r^{-1/2})} 
		& \geq
		\frac{2 \pi^{d/2}}{\Gamma(d/2)} 
		\cdot
		\frac{p(t, \sqrt{br})}{t\ell(r^{-1/2})}
		\cdot
		\frac{r^{d/2+1}(b^{d/2+1}-a^{d/2+1})}{d+2} \\
		\frac{Q_t(br)-Q_t(ar)}{t\ell(r^{-1/2})}
		& \leq
		\frac{2 \pi^{d/2}}{\Gamma(d/2)} 
		\cdot
		\frac{p(t, \sqrt{ar})}{t\ell(r^{-1/2})}
		\cdot
		\frac{r^{d/2+1}(b^{d/2+1}-a^{d/2+1})}{d+2}.
		\end{aligned}
	\end{equation}
	Thanks to Proposition \ref{prop:1}, we may use Theorem \ref{thm:1} to get
	\begin{align*}
		\lim_{\atop{r \to 0^+}{t \psi(r^{-1/2}) \to 0}}
		\frac{Q_t(br)-Q_t(ar)}{r t\ell (r^{-1/2})}=\frac{b-a}{2}.
	\end{align*}
	Hence, by the first inequality in \eqref{eq:15}
	\begin{align*}
		\limsup_{\atop{r \to 0^+}{t \psi(r^{-1/2}) \to 0}} 
		\frac{p(t, \sqrt{br})}{r^{-d/2} t \ell (r^{-1/2})} 
		\leq
		\frac{1}{2}
		\cdot 
		\frac{\Gamma(d/2)}{2 \pi^{d/2}} (d+2)
		\frac{b-a}{b^{d/2+1}-a^{d/2+1}}.
	\end{align*}
	By taking $b=1$, $a=1-\epsilon $ and letting $\epsilon$ to zero we obtain
	\begin{align*}
		\lim_{\atop{r \to 0^+}{t \psi(r^{-1/2}) \to 0}}
		\frac{p(t, r^{1/2})} {r^{-d/2} t \ell (r^{-1/2})}
		\leq 
		\frac{\Gamma(d/2)}{2 \pi^{d/2}}. 
	\end{align*}
	Similarly, using the second inequality in \eqref{eq:15}, we show that
	\[
		\liminf_{\atop{r \to 0^+}{t \psi(r^{-1/2}) \to 0}}
		\frac{p(t, r^{1/2})}{r^{-d/2} t \ell (r^{-1/2})} 
		\geq 
		\frac{\Gamma(d/2)}{2 \pi^{d/2}}.
		\qedhere
	\]
\end{proof}

\begin{theorem}
	\label{thm:5}
	Let $\mathbf{X}$ be an isotropic unimodal L\'{e}vy process on $\RR^d$ with the L\'{e}vy--Khintchine
	exponent $\psi$ and the L\'{e}vy density $\nu$. Let $\ell \in \calR_0^\infty$. The following statements 
	are equivalent: 
	\begin{enumerate}
		\item $\psi \in \Pi_\ell^\infty$;
		\item there is $c > 0$,
		\[
			\lim_{\atop{x \to 0}{t \psi(\norm{x}^{-1}) \to 0}} \frac{p(t, x)}{\norm{x}^{-d} t 
			\ell\big(\norm{x}^{-1}\big)} = c;
		\] 
		\item there is $c > 0$,
		\[
			\lim_{x \to 0} \frac{\nu(x)}{\norm{x}^{-d} \ell\big(\norm{x}^{-1}\big)} = c.
		\]
	\end{enumerate} 
\end{theorem}
\begin{proof}
	The implication (i)~$\Rightarrow$~(ii) follows by Theorem \ref{thm:2}. Next, (ii)~$\Rightarrow$~(iii) is
	a consequence of the fact that
	\[
		\lim_{t \to 0^+} t^{-1} p(t, x) = \nu(x),
	\]
	vaguely on $\RR^d \setminus \{0\}$. To prove that (iii) implies (i), first we consider the case $d = 1$. By
	\eqref{eq:30}, we can write
	\[
		\psi(x) 
		=
		\int_\RR \big(1 - \cos (x t)\big) \nu(t) {\: \rm d}t
		=
		\frac{2}{x} \int_0^\infty \big(1 - \cos t\big) \nu(t/x) {\: \rm d}t,
	\]
	thus, for any $A, \lambda, x > 0$
	\begin{equation}
		\label{eq:21}
			\frac{\psi(\lambda x) - \psi(x)}{2 \ell(x)}
			=
			\int_0^A \big(\cos t - \cos \lambda t\big) \frac{\nu(t/x)}{x \ell(x)} {\: \rm d} t
			+
			\int_A^\infty \big(\cos t - \cos (\lambda t) \big) \frac{\nu(t/x)}{x \ell(x)} {\: \rm d} t.
	\end{equation}
	We observe that, by \eqref{eq:14} and the dominated convergence,
	\[
		\lim_{x \to \infty}
		\int_0^A \big(1 - \cos (\lambda t)\big) \frac{\nu(t/x)}{x \ell(x)} {\: \rm d}t
		=
		c \int_0^A \big(1 - \cos (\lambda t)\big) \frac{{\rm d} t}{t}=c \int_0^{\lambda A} \big(1 - \cos t \big) \frac{{\rm d} t}{t}.
	\]
	Hence,
	\[
		\lim_{x \to \infty} \int_0^A \big(\cos t - \cos (\lambda t)\big) \frac{\nu(t/x)}{x \ell(x)}
		{\: \rm d}t
		=
		c \int_A^{\lambda A} \big(1 - \cos t\big) \frac{{\rm d} t}{t} = c \log \lambda -c \int_A^{\lambda A} \cos t\frac{{\rm d} t}{t} .
	\]
	Applying the second mean value theorem we obtain 
	$$	\bigg|\int_A^{\lambda A} \cos t\frac{{\rm d} t}{t}	\bigg|\le 2\left(1+\frac 1\lambda\right)\frac1A.$$

	To deal with the second integral in \eqref{eq:21}, we use monotonicity of $\nu$. Namely, we have $\nu(s)=\int^\infty_s\mu(du)$ for some non-negative measure  $\mu$. Hence, by the Fubini theorem,
	\begin{align*}
		\bigg|
		\int_A^\infty \cos(\lambda t) \frac{\nu(t/x)}{x \ell(x)} {\: \rm d} t
		\bigg|
		& = \frac{1}{ x\ell(x)}\bigg|\int^\infty_{A/x}\int^{xu}_A\cos(\lambda t){\: \rm d}t \mu({\rm d}u)\bigg|\leq \frac{2}{\lambda x\ell(x)}\int^\infty_{A/x}\mu({\rm d}u)
		=
		\frac{2}{\lambda} \frac{\nu(A/x)}{x \ell(x)}.
	\end{align*}
	Finally, we get
	\[
		\limsup_{x \to \infty}
		\bigg|
		\int_A^\infty \big(\cos t - \cos \lambda t\big) \frac{\nu(t/x)}{x \ell(x)} {\: \rm d}t
		\bigg|
		\leq
		\frac{2 c}{A} \bigg(1 + \frac{1}{\lambda}\bigg).
	\]
	Since $A$ was arbitrary, we conclude that
	\[
		c \log \lambda\le \liminf_{x \to \infty}
		\frac{\psi(\lambda x) - \psi(x)}{2 \ell(x)} \le \limsup_{x \to \infty}
		\frac{\psi(\lambda x) - \psi(x)}{2 \ell(x)} \leq c \log \lambda,
	\]
	which finishes the proof for $d = 1$.

	For $d \geq 2$, we consider the L\'evy measure $\nu_1$ corresponding to the one-dimensional projection of
	$\mathbf{X}$. We claim that
	\begin{claim}
	\[
		\lim_{x \to 0} \frac{\nu_1(x)}{\abs{x}^{-1} \ell\big(\abs{x}^{-1}\big)} = c \frac{\pi^{d/2}}{\Gamma(d/2)}.
	\]
	\end{claim}
	Indeed, using spherical coordinates we can write
	\begin{align*}
		\nu_1(x) 
		& = \int_{\RR^{d-1}} \nu\Big(\sqrt{\norm{u}^2 + x^2}\Big) {\: \rm d} u \\
		& = \omega_{d-1} \int_0^\infty \nu\big(\sqrt{r^2 + x^2}\big) r^{d-2} {\: \rm d} r
	\end{align*}
	where $\omega_{d-1}$ is the surface measure of the unite sphere in $\RR^{d-1}$. Since
	\[
		\omega_{d-1} \int_1^{\infty} \nu\big(\sqrt{r^2 + x^2}\big) r^{d-2} {\: \rm d} r
		\leq
		\nu\big(\{u \in \RR^d : \norm{u} \geq 1\big\}\big),
	\]
	it is enough to show that
	\[
		\lim_{x \to 0} \frac{\omega_{d-1}}{x^{-1} \ell\big(x^{-1}\big)} 
		\int_0^1 \nu\big(\sqrt{r^2 + x^2}\big) r^{d-2} {\: \rm d} r
		=
		\frac{\pi^{d/2}}{\Gamma(d/2)}.
	\]
	By the change of variables, we have
	\[
		\int_0^1 \nu\big(\sqrt{r^2 + x^2}\big) r^{d-2} {\: \rm d} r 
		=
		\int_0^{1/\abs{x}} \abs{x}^{d-1} \nu\big(\abs{x} \sqrt{1+r^2}\big) r^{d-2} {\: \rm d} r.
	\]
	Since 
	\begin{align*}
		& \abs{x}^{d-1} \frac{\nu\big(\abs{x} \sqrt{1 + r^2}\big)}{\abs{x}^{-1} \ell\big(x^{-1}\big)} r^{d-2} \\
		& \qquad\qquad\qquad =
		\frac{\nu\big(\abs{x}\sqrt{1 + r^2} \big)}{\abs{x}^{-d} (1 + r^2)^{-d/2} \ell\big(\abs{x}^{-1}(1+r^2)^{-1/2}\big)}
		\cdot \frac{\ell\big(\abs{x}^{-1} (1 + r^2)^{-1/2}\big)}{\ell\big(\abs{x}^{-1}\big)} 
		\cdot \frac{r^{d-2}}{(1+r^2)^{d/2}},
	\end{align*}
	for a fixed $r > 0$ we obtain
	\[
		\lim_{x \to 0} \abs{x}^{d-1} \frac{\nu\big(\abs{x} \sqrt{1 + r^2}\big)}{\abs{x}^{-1} \ell\big(x^{-1}\big)} r^{d-2}
		=
		c \frac{r^{d-2}}{(1 + r^2)^{d/2}}.
	\]
	Therefore, by \eqref{eq:14} we can use the dominated convergence theorem to conclude that
	\begin{align*}
		\lim_{x \to 0} \frac{\omega_{d-1}}{x^{-1} \ell\big(x^{-1}\big)}
		\int_0^1 \nu\big(\sqrt{r^2 + x^2}\big) r^{d-2} {\: \rm d}r
		& =
		\omega_{d-1} 
		\int_0^{\infty}
		\frac{r^{d-2}}{(1 + r^2)^{d/2}}
		{\: \rm d}r \\
		& =
		c \omega_{d-1} \frac{\Gamma((d-1)/2) \Gamma(1/2)}{2 \Gamma(d/2)}
		= 
		c \frac{\pi^{d/2}}{\Gamma(d/2)}. \qedhere
	\end{align*}
\end{proof}
By the same line of reasoning as in the proofs of Proposition \ref{prop:1} and  Theorems \ref{thm:1}, \ref{thm:2} and \ref{thm:5}
one can show the corresponding results if the L\'{e}vy-Khintchine exponent belongs to de Haan class at the origin. Only one modification is needed to prove that (iii) implies (i). Namely, one should consider an asymptotically equal (at the origin)  L\'{e}vy-Khintchine exponent corresponding to the  L\'{e}vy density $\tilde{\nu}=\nu|_{B^c_1}$ (compare with the proof of \cite[Theorem 7]{cgt}). 
\begin{theorem}
	\label{thm:5Inf}
	Let $\mathbf{X}$ be an isotropic unimodal L\'{e}vy process on $\RR^d$ with the L\'{e}vy--Khintchine
	exponent $\psi$ and the L\'{e}vy density $\nu$. Let $\ell \in \calR_0^0$. The following statements 
	are equivalent: 
	\begin{enumerate}
		\item $\psi \in \Pi_\ell^0$;
		\item there is $c > 0$,
		\[
			\lim_{\atop{|x| \to \infty}{t \psi(\norm{x}^{-1}) \to 0}} \frac{p(t, x)}{\norm{x}^{-d} t 
			\ell\big(\norm{x}^{-1}\big)} = c;
		\] 
		\item there is $c > 0$,
		\[
			\lim_{|x| \to \infty} \frac{\nu(x)}{\norm{x}^{-d} \ell\big(\norm{x}^{-1}\big)} = c.
		\]
	\end{enumerate} 
\end{theorem}
\subsection{Large $t \psi(\norm{x}^{-1})$}
In the case of large $t\psi(\norm{x}^{-1})$, we again use $Q_t$ given by \eqref{eq:16}, but this time with
$q_t(r) = t \ell(\sqrt{r}) e^{-t \psi(\sqrt{r})}$ and $w_t(r) = t \psi(\sqrt{r})$. In this section we assume that $\ell$
is a \emph{bounded} function slowly varying at infinity. We start by the following observation.
\begin{lemma}
	\label{lem:1}
	Suppose $f: [x_0, \infty) \rightarrow \RR$, for some $x_0 \geq 0$, is such that
	\[
		\sup_{x_0 < x \leq y \leq 2x} \big|f(x) - f(y) \big| < \infty.
	\]
	Then for each $a \geq 1$ there are $C, \delta > 0$ such that for all $r, t > 0$, if $x, rx > x_0$ then
	\[
		e^{t f(x)}
		\big|
		e^{-t f(r x)} - e^{-t f(rax)}
		\big|
		\leq
		t C^t \max\{r, r^{-1}\}^{\delta t} 
		\big|f(rx) - f(r a x) \big|.
	\]
\end{lemma}
\begin{proof}
	Let
	\[
		A = \sup_{x_0 < x \leq y \leq 2x} \big|f(x) - f(y) \big|.
	\]
	and $\delta = \frac{A}{\log 2}$. First, we show, that for all $x, y > x_0$
	\begin{equation}
		\label{eq:41}
		e^{f(x) - f(y)} \leq e^A \max\{x/y, y/x\}^\delta.
	\end{equation}
	Fix $x > x_0$. If $\lambda \in [1, 2]$ then
	\[
		\abs{f(\lambda x) - f(x)} \leq A.
	\]
	Therefore, if $2^k < \lambda \leq 2^{k+1}$ for $k \geq 1$ then
	\begin{align*}
		\abs{f(\lambda x) - f(x)}
		& \leq \sum_{j = 1}^k \abs{f(2^j x) - f(2^{j-1} x)} 
		+ \abs{f(2^k x) - f(\lambda x)} \\
		& \leq
		A (1 + k).
	\end{align*}
	Next, we observe that for all $x, rx > x_0$
	\[
		e^{t f(x)} 
		\big|
		e^{-tf(rx)} - e^{-t f(r a x)}
		\big|
		\leq
		t\big(e^{t f(x) - tf(rx)} + e^{t f(x) - tf(r a x)} \big) 
		\big|
		f(r x) - f(r a x)
		\big|.
	\]
	Therefore, by \eqref{eq:41}, there is $C > 0$ such that for all $t > 0$
	\[
		e^{t f(x) - tf(r x)} + e^{t f(x) - tf(rax)} \leq C^t \max\{r, r^{-1}\}^{\delta t},
	\]
	which finishes the proof.
\end{proof}

\begin{proposition}
	\label{prop:2}
	Suppose $\psi \in \Pi^\infty_\ell$ for some bounded $\ell \in \calR_0^\infty$.
	Then there are $C,\lambda_0, \epsilon > 0$ such that for all $x > 0$ and $t \in (0, \epsilon)$
	\begin{equation}
		\label{eq:19}
		\frac{Q_t(\lambda)}{\lambda}
		\leq
		\begin{cases}
		C t \ell(\lambda^{-1/2}) e^{-t\psi(\lambda^{-1/2})} & \text{if} \quad 0 < \lambda \leq \lambda_0,\\
		C t \ell(\lambda_0^{-1/2}) e^{-t\psi(\lambda_0^{-1/2})} & \text{otherwise.}
		\end{cases}
	\end{equation}
	In particular,
	\begin{equation}
		\label{eq:20}
		\lim_{\atop{t \to 0^+}{t\psi(\sqrt{\lambda}) \to \infty}}
		\frac{\lambda \calL\{{\rm d} Q_t\}(\lambda)}{t \ell(\sqrt{\lambda}) e^{-t\psi(\sqrt{\lambda})}}
		=\frac{1}{2}.
	\end{equation}
\end{proposition}
\begin{proof}
	The proof follows the same line as of Proposition \ref{prop:1}. Let $a\ge 1$. First, we show that 
	for every $\eta \in (0, 1)$ there are $C, x_0, \epsilon > 0$ such that for all $r, \lambda > 0$
	and $t \in (0, \epsilon)$, if $\sqrt{\lambda} > x_0$ then
	\begin{equation}
		\label{eq:22}
		\Big|
		e^{-t\psi(r \sqrt{\lambda})} - e^{-t\psi(r \sqrt{\lambda a})}
		\Big|
		\leq
		C t \ell(\sqrt{\lambda}) e^{-t\psi(\sqrt{\lambda})} \max\{r, r^{-1}\}^{\eta}.
	\end{equation}
	Observe that there are $C, x_0 > 0$ such that for all $r, \lambda > 0$ if $\sqrt{\lambda} > x_0$ then
	\begin{equation}
		\label{eq:43}
		\big|\psi(r \sqrt{\lambda}) - \psi(r \sqrt{\lambda a}) \big|
		\leq
		C
		\ell(\sqrt{\lambda}) \max\{r, r^{-1}\}^{\eta/2}.
	\end{equation}
	Indeed, if $\sqrt{\lambda}, r \sqrt{\lambda} > x_0$ then \eqref{eq:43} is a consequence of \eqref{eq:5} and
	\eqref{eq:14}. If $\sqrt{\lambda} > x_0 \geq r \sqrt{\lambda}$ then by \eqref{eq:31} we can estimate
	\[
		\big|\psi(r \sqrt{\lambda}) - \psi(r \sqrt{\lambda a}) \big|
		\leq
		2 \psi^*(r \sqrt{\lambda a}),
	\]
	and since by \eqref{eq:14}
	\[
		\ell(x_0) \leq 2 \ell(\sqrt{\lambda}) r^{-\eta/2},
	\]
	we get \eqref{eq:43}.
	
	Next, for $\psi \in \Pi_\ell^\infty$ we apply \cite[Theorem 3.8.6(b)]{bgt} to show that
	\[
		\sup_{0 < x \leq y \leq 2x} \big| \psi(x) - \psi(y) \big| < \infty.
	\]
	Therefore, by Lemma \ref{lem:1}, there are $C, \delta > 0$ such that for all $r, t, \lambda > 0$,
	\[
		e^{t \psi(\sqrt{\lambda})}
		\big|
		e^{-t \psi(r \sqrt{\lambda})} - e^{-t\psi(r \sqrt{\lambda a})}
		\big|
		\leq
		t C^{t+1} \max\{r, r^{-1}\}^{\delta t}
		\big|
		\psi(r \sqrt{\lambda}) - \psi(r \sqrt{\lambda a})
		\big|,
	\]
	thus by \eqref{eq:43}, whenever $\sqrt{\lambda} > x_0$ we obtain
	\[
		e^{t \psi(\sqrt{\lambda})}
 		\big|
		e^{-t \psi(r \sqrt{\lambda})} - e^{-t\psi(r \sqrt{\lambda a})}
		\big|
		\leq
		t C^{t+1} \ell(\sqrt{\lambda}) \max\{r, r^{-1}\}^{\delta t + \eta/2}.
	\]
	By taking $t \in (0, \epsilon)$ for $\epsilon$ sufficiently small to satisfy $2 \delta \epsilon < 2 - \eta$, we 
	conclude \eqref{eq:22}.

	Now, let $\eta = 1/2$. By applying \eqref{eq:22} to the formula \eqref{eq:12} we get
	\begin{equation}
		\label{eq:28}
		\calL\{{\rm d} U_t\}(\lambda) - \calL\{{\rm d} U_t\}(\lambda a)
		\leq
		C
		t \ell(\sqrt{\lambda}) e^{-t \psi(\sqrt{\lambda})}
	\end{equation}
	for all $t \in (0, \epsilon)$ and $\sqrt{\lambda} \geq x_0$.
	
	Again, with the help of  \eqref{eq:22} we can see that the integrand in \eqref{eq:12} is uniformly
	bounded by an integrable function. Since for a fixed $r > 0$,
	\begin{align*}
		\lim_{\atop{t \to 0^+}{t \psi(\sqrt{\lambda}) \to \infty}}
		\frac{e^{-t\psi(r\sqrt{\lambda a})}-e^{-t\psi(r\sqrt{\lambda})}}
		{ t \ell(\sqrt{\lambda}) e^{-t \psi(\sqrt{\lambda})}}
		& =
		\lim_{\atop{t \to 0^+}{t \psi(\sqrt{\lambda}) \to \infty}}
		e^{t(\psi(\sqrt{\lambda}) - \psi(r \sqrt{\lambda a}))}
		\frac{\psi(r \sqrt{\lambda a}) - \psi(r \sqrt{\lambda})}{\ell(\sqrt{\lambda})} \\
		&= \frac{\log a}{2},
	\end{align*}
	by the dominated convergence theorem we obtain the following claim.
	\begin{claim}
		For $a \geq 1$
		\begin{equation}
			\label{eq:24}
			\lim_{\atop{t \to 0^+}{t \psi(\sqrt{\lambda}) \to \infty}}
			\frac{\calL\{{\rm d} U_t\}(\lambda a) - \calL\{{\rm d} U_t\}(\lambda)}
			{t \ell(\sqrt{\lambda}) e^{-t\psi(\sqrt{\lambda})}}
			=
			\frac{\log a}{2}.
		\end{equation}
	\end{claim}
	Now, the formulas \eqref{eq:11} together with \eqref{eq:24} imply \eqref{eq:20}. Finally, let us recall
	the estimate \eqref{eq:23}
	\[
		Q_t(x) \leq 2 e x \left(\calL\{{\rm d} U_t\}(2^{-1} x^{-1}) - \calL\{{\rm d} U_t\}(x^{-1})\right).
	\]
	If $2 x \leq x_0^{-2}$, we may use \eqref{eq:28} to get
	\[
		Q_t(x) \leq C x t \ell(x^{-1/2}) e^{-t\psi(x^{-1/2})}.
	\]
	For $2 x > x_0^{-2}$, by \eqref{eq:31} we estimate
	\[
		e^{t\psi(x_0)}
		\Big\lvert
		e^{-t \psi(r x^{-1/2})} - e^{-t\psi(r x^{-1/2} 2^{-1/2})} 
		\Big\rvert
		\leq
		C t \psi^*(r x^{-1/2})
		\leq
		C t (1+r^2).
	\]
	Hence, by \eqref{eq:12} we obtain
	\begin{align*}
		Q_t(x) \leq C x t \ell(x_0) e^{-t\psi(x_0)}
	\end{align*}
	which finishes the proof.
\end{proof}

\begin{theorem}
	\label{thm:3}
	Let $\mathbf{X}$ be an isotropic unimodal L\'{e}vy process on $\RR^d$ with the L\'{e}vy--Khintchine
	exponent $\psi \in \Pi^\infty_\ell$ for some bounded $\ell \in \calR_0^\infty$. Then
	\begin{equation}
		\label{eq:25}
		\lim_{\atop{t \to 0^+}{t\psi(\norm{x}^{-1}) \to \infty}}
		\frac{p(t, x)}{t \norm{x}^{-d} \ell(\norm{x}^{-1}) e^{-t \psi(\norm{x}^{-1})}}
		=
		\frac{\Gamma(d/2)}{2 \pi^{d/2}}.
	\end{equation} 
\end{theorem}
\begin{proof}
	The argument is analogous to the proof of Theorem \ref{thm:2}. One needs to use Proposition \ref{prop:2}
	and Theorem \ref{thm:8} instead of Proposition \ref{prop:1} and Theorem \ref{thm:1}.
\end{proof}

\subsection{Green function asymptotics}
In this section we study the potential measure $G$ associated to a \emph{transient} isotropic unimodal L\'{e}vy process
$\mathbf{X}$, that is 
\begin{align*}
	G(x, A)=\int _{0}^{\infty }\PP^x(X_t\in A) {\: \rm d}t
\end{align*}  
where $\PP^x$ is the standard measure $\mathbb{P}(\: \cdot \: | X_0 =x)$ and $A \subset \RR ^d$ is a Borel set.
We set $G(A)=G(0,A)$. We also use the same notation $G$ for the density of the part of the potential measure which
is absolutely continuous with respect to the Lebesgue measure. We have $G(x,y)= G(0,y-x)$ and we write $G(x)=G(0,x)$.
In the following theorem we provide the asymptotic behaviour of the potential measure $G$.
\begin{theorem}
	\label{GreenAsymp} 
	Let $\mathbf{X}$ be an isotropic unimodal L\'{e}vy process on $\RR^d$ with
	the L\'{e}vy--Khintchine exponent $\psi \in \Pi^\infty_\ell$ for some $\ell \in \calR_0^\infty$. Then for all
	$\lambda > 1$
	\begin{align*}
		\lim _{r \to 0^+}
		\frac{G\big(\{x : r \leq \norm{x} \leq \lambda r\}\big)}
		{\psi(r^{-1/2})^{-2} \ell (r^{-1})}
		= \frac{\log \lambda}{2}.
	\end{align*}
\end{theorem}
\begin{proof}
	Let
	\[ 
		f(r) 
		=G\big(\{x : \norm{x} \leq \sqrt{r}\}\big)
		= \int_0^\infty \int_{\norm{x} \leq \sqrt{r}} p(t, {\rm d} x) {\: \rm d}t.
	\]
	Hence, by the Fubini--Tonelli's theorem
	\[
		\lambda \calL f(\lambda)
	 	=
		\int_0^{\infty} \int_{\RR^d} e^{-\lambda \norm{x}^2} p(t, {\rm d} x) {\: \rm d}t.
	\]
	By \eqref{eq:33}, the Fubini--Tonelli's theorem and integration in polar coordinates give
	\begin{align}
		\nonumber
		\lambda \calL f(\lambda) 
		&= \int_0^{\infty} 
		\int_{\RR^d} e^{-t \psi(\xi \sqrt{\lambda})} e^{-\norm{\xi}^2/4} {\: \rm d}\xi 	{\: \rm d}t \\
		\nonumber
		&= (4 \pi)^{-d/2} \int_{\RR^d} e^{-\norm{\xi}^2/4} \frac{{\rm d}\xi}{\psi(\xi \sqrt{\lambda})} \\
		\label{eq:27}
		&= \frac{2^{1-d}}{\Gamma(d/2)} \int_0^{\infty} e^{-r^2/4} r^{d-1} \frac{{\rm d}r}{\psi(r \sqrt{\lambda})}.
	\end{align}
	Hence,
	\begin{equation}
		\label{eq:40}
		s \calL f(s )- \lambda s \calL f(\lambda s)
		= \frac{2^{1-d}}{\Gamma (d/2)}\int _{0}^\infty r^{d-1}e^{-r^2/4}
		\left(\frac{1}{\psi(r \sqrt{s})} - \frac{1}{\psi(r \sqrt{\lambda s})} \right)
		{\: \rm d}r.
	\end{equation}
	Next, by \eqref{eq:14} and \eqref{eq:5}, for any $\delta \in (0, 1/3)$  there are $x_0, C > 0$  such that if
	$r \sqrt{s}, \sqrt{s} \geq x_0$ then
	\begin{align}
	\label{eq:36}
		\big| \psi(r \sqrt{s \lambda}) - \psi(r \sqrt{s}) \big|
		& \leq
		C \ell(\sqrt{s}) \max\{r, r^{-1}\}^\delta,
	\end{align}
	and
	\begin{equation}
		\label{eq:38}
		\begin{aligned}
			\psi(\sqrt{s}) & \leq C \psi(r \sqrt{s}) \max\{r, r^{-1}\}^\delta, \\
			\psi(\sqrt{s}) & \leq C \psi(r \sqrt{s \lambda}) \max\{r, r^{-1}\}^\delta.
		\end{aligned}
	\end{equation}
	For $r \leq x_0/\sqrt{s}$, we have
	\begin{equation}
		\label{eq:44}
		\int_{0}^{x_0/\sqrt{s}} 
		r^{d-1} e^{-r^2/4}
		\frac{1}{\psi(r \sqrt{s})}
		{\: \rm d}r
		\le
		s^{-d/2}
		\int _{0}^{x_0} r^{d-1}
		\frac{1}{\psi(r)} {\: \rm d}r.
	\end{equation}
	Since the process is transient the right-hand side is finite. Because $\psi^2/\ell \in \calR_0^\infty$, by
	\eqref{eq:44} we get
	\begin{equation}
		\label{eq:45}
		\lim_{s \to \infty}
		\frac{\psi(\sqrt{s})^2}{\ell(\sqrt{s})} 
		\int _{0}^{x_0/\sqrt{s}} 
		r^{d-1}e^{-r^2/4}
		\left(\frac{1}{\psi(r \sqrt{s})}- \frac{1}{\psi(r \sqrt{\lambda s})}\right)
		{\: \rm d}r=0.
	\end{equation}
	If $r \sqrt{s} \geq x_0$, by \eqref{eq:36} and \eqref{eq:38}, we can estimate
	\begin{equation}
		\label{eq:42}
		\begin{aligned}
		\frac{\psi(\sqrt{s})^2}{\ell(\sqrt{s})} 
		\Big| 
		\frac{1}{\psi(r \sqrt{s})} - \frac{1}{\psi(r \sqrt{s \lambda})}
		\Big|
		& \leq
		\frac{\big| \psi(r \sqrt{s \lambda}) - \psi(r \sqrt{s})\big|}{\ell(\sqrt{s})}
		\cdot
		\frac{\psi(\sqrt{s})}{\psi(r \sqrt{s})}
		\cdot
		\frac{\psi(\sqrt{s})}{\psi(r \sqrt{s \lambda})} \\
		& \leq
		C
		\max\{r, r^{-1}\}^{3 \delta}. 
		\end{aligned}
	\end{equation}
	In particular, the integrand in \eqref{eq:40} restricted to $[x_0/\sqrt{s}, \infty)$ is bounded by
	an integrable function. Since $\psi \in \calR_0^\infty$, by \eqref{eq:5}
	\[
		\lim_{s \to \infty}
		\frac{\psi(\sqrt{s})^2}{\ell(\sqrt{s})} 
		\left(\frac{1}{\psi(r \sqrt{s})} - \frac{1}{\psi(r \sqrt{s \lambda})}\right)
		=
		\frac{\log \lambda}{2},
	\]
	thus by the dominated convergence theorem and \eqref{eq:45} we obtain
	\begin{align*}
		\lim_{s \to \infty}
		\frac{s \calL f(s )-\lambda s \calL f(\lambda s)}{\ell(\sqrt{s})\psi(\sqrt{s})^{-2}} 
		= \frac{\log \lambda}{2}.
	\end{align*}
	Now, by applying de Haan Tauberian theorem \cite[Theorem 3.9.1]{bgt} we conclude that
	\[
		\lim_{r \to 0^+}
		\frac{f(\lambda r) - f(r)}{\ell(r^{-1/2}) \psi(r^{-1/2})^{-2}} =\frac{\log \lambda}{2}.
		\qedhere
	\]
\end{proof}

\begin{corollary}
	\label{cor:2}
	Assume that $\psi \in \Pi_\ell^\infty$ for some $\ell \in \calR^\infty_0$. Then
	\begin{align*}
		\lim _{x \to 0} 
		\frac{G(x)}{\norm{x}^{-d} \psi(\norm{x}^{-1})^{-2} \ell(\norm{x}^{-1})}
		=
		\frac{\Gamma(d/2)}{2 \pi^{d/2}}.
	\end{align*}
\end{corollary}
\begin{proof}
	We use Theorem \ref{GreenAsymp} and the line of reasons from Theorem \ref{thm:2}.
\end{proof}
Similarly one can prove the following proposition.
\begin{proposition}
	\label{cor:2I}
	Assume that $\psi \in \Pi_\ell^0$ for some $\ell \in \calR^0_0$. Then
	\begin{align*}
		\lim _{|x| \to \infty} 
		\frac{G(x)}{\norm{x}^{-d} \psi(\norm{x}^{-1})^{-2} \ell(\norm{x}^{-1})}
		=
		\frac{\Gamma(d/2)}{2 \pi^{d/2}}.
	\end{align*}
\end{proposition}

In the rest of this section we study the asymptotic behaviour of the Green function. The proofs below  rely on Theorem
\ref{thm:equiG}, but  for future reference,  it is  convenient to include them   in this section.

Let us recall that for any integrable radial function $g:\RR^d \rightarrow \RR$
\begin{equation}
	\label{eq:Hankel}
	\int_{\RR^d} e^{i\sprod{\xi}{x}}g(x) {\: \rm d} x
	= 
	|\xi|^{1-d/2} (2\pi)^{d/2} \int^\infty_0 s^{d/2}g(s)J_{d/2-1}(|\xi|s){\: \rm d}s,
\end{equation}
where $J_\alpha$ is the Bessel function of the first kind defined for $\alpha > -1/2$ and $x > 0$ by
the complex integral
\[
	J_{\alpha}(x) = \frac{1}{2 \pi i} 
	\int_{\abs{z} = 1} e^{x(z - z^{-1})/2} z^{-\alpha-1} {\: \rm d}z.
\]
For each $\lambda > 0$, let us denote by $G^\lambda$ a $\lambda$-resolvent kernel, that is 
\[
	G^\lambda(x) = \int_0^\infty e^{-\lambda t} p(t, x) {\: \rm d}t
\]
for $x \in \RR^d$. For any Borel set $A \subset \RR^d$ we set
\[
	G^\lambda(A) = \int_0^\infty \int_A e^{-\lambda t} p(t, {\rm d} x) {\: \rm d}t.
\]
\begin{theorem}
	\label{thm:6}
	Suppose $d\geq 6$. Let $\mathbf{X}$ be an isotropic unimodal L\'{e}vy process on $\RR^d$ with
	the L\'{e}vy--Khintchine exponent $\psi$. Then 
	\begin{equation}
		\label{eq:limGe}
		\lim_{x \to 0} \frac{G(x)}{|x|^{-d} \psi(|x|^{-1})^{-1}} = c>0,
	\end{equation}
	if and only if $\psi\in\RInf$, for some $\alpha >0 $. In particular, \eqref{eq:limGe} implies that
	\[
		c=
		2^{-\alpha}\pi^{d/2}\frac{\Gamma\left((d-\alpha)/2\right)}{\Gamma\left(\alpha/2\right)}.
	\]
\end{theorem}
\begin{proof}
	If $\psi\in\RInf$, for some $\alpha>0$, then by \cite[Corollary 4]{cgt} we obtain \eqref{eq:limGe}. 
	Hence, it is enough to prove the converse. Assume that \eqref{eq:limGe} holds. We are going to use a version of
	Drasin--Shea--Jordan theorem proved in \cite[Theorem 1]{MR1760604}. First, let us observe that, by \eqref{eq:31}, $G$
	is unbounded. Next, for each $t > 0$, in view of Theorem \ref{thm:equiG} the characteristic function of $X_t$
	is integrable. Hence, following the argument given in \cite[Lemma III.3]{MR1054115} we obtain
	\[
		\lim_{x\to 0}\frac{G^1(x)}{G(x)}=1.
	\]
	Since 
	\[
		\lim_{x\to 0}\frac{\psi(|x|^{-1})}{\psi(|x|^{-1})+1}=1,
	\]
	we also have
	\begin{equation}
		\label{eq:limGe1}
		\lim_{x\to 0} G^1(x)|x|^d\left(\psi\left(|x|^{-1}\right)+1\right)=c.
	\end{equation}
	Moreover, $G^1$ is integrable and its Fourier transform equals to $(\psi(\xi)+1)^{-1}$. Thus, by \eqref{eq:Hankel},
	we can write
	\begin{align*}
		\frac{\norm{\xi}^{d/2-1}}{\psi(\xi)+1}
		& =  \norm{\xi}^{d/2-1}\int_{\RR^d} G^1(x) e^{-i \sprod{x}{\xi}} {\: \rm d}\xi \\
		& = (2\pi)^{d/2}\int^\infty_0 s^{d/2} G^1(s) J_{d/2-1}(\norm{\xi}s){\: \rm d}s \\
		& = (2\pi)^{d/2}\int^\infty_0 s^{-d/2-1}G^1(1/s)J_{d/2-1}(\norm{\xi}/s)\frac{{\rm d}s}{s}.
	\end{align*}
	If we set
	\[
		f(r) = r^{-d/2-1} G^1(1/r),
	\]
	then for $r > 0$
	\[
		(2\pi)^{-d/2} \frac{r^{d/2-1}}{\psi(r) + 1}	
		=
		\calM(J_{d/2-1},f)(r),
	\]
	where for two functions $f,g : [0, \infty) \rightarrow \CC$, by $\calM(f, g)$ we denote their Mellin convolution,
	that is
	\[
		\calM(f, g)(r) = \int_0^\infty f(r/ s) g(s) \frac{{\rm d} s}{s}.
	\]
	Now, in view of \eqref{eq:limGe1}, we have
	\[
		\lim_{r \to \infty} \frac{\calM(J_{d/2-1},f)(r)}{f(r)}=(2\pi)^{-d/2} c^{-1}.
	\]
	Since $G^1$ is integrable and non-increasing the function $s \mapsto s^d G^1(s)$ is bounded, thus
	\[
		\lim_{r\to 0^+} f(r)=0.
	\]
	Let $\check{J}_{d/2-1}$ be the Mellin transform of $J_{d/2-1}$, i.e. for $z \in \CC$
	\begin{equation}
		\label{eq:17}
		\check{J}_{d/2-1}(z)
		=
		\int_0^\infty
		t^{-z-1} J_{d/2-1}(t) {\: \rm d} t,
	\end{equation}
	whenever the integral converges. By the well-known asymptotics for the Bessel functions of the first kind,
	there is $C > 0$ such that for all $r > 0$
	\begin{equation}
		\label{eq:50}
		0 \leq J_{d/2-1}(r) \leq C \min\big\{r^{-1/2}, r^{d/2-1} \big\}.
	\end{equation}
	Hence, the integral \eqref{eq:17} is absolutely convergent on the strip $\{z \in \CC : -1/2 < \Re z < d/2-1\}$.
	Moreover, by \cite[13.24, (1)]{MR0010746}, we can calculate
	\[ 
		\check{J}_{d/2-1}(z) =
		\frac{2^{-z-1}\Gamma\left((d-2-2z))/4\right)}{\Gamma\left((d+2+2z))/4\right)}.
	\]
	Finally, in view of Theorem \ref{thm:equiG} and \eqref{eq:31}, the function $f$ has bounded decrease. For detailed
	study of functions with bounded decrease we refer to \cite[Section 2.1]{bgt}. Next, we calculate
	\[
		\rho = \limsup_{r \to \infty} \frac{\log f(r)}{\log r} \in \big[d/2-3, d/2-1\big).
	\]
	Therefore, by \cite[Proposition 1]{MR1476406} and \cite[Section 5]{MR1710167}, the hypothesis 
	of \cite[Theorem 1 and Theorem 2]{MR1760604} are satisfied and the theorem follows.
\end{proof}

\begin{theorem}
	\label{thm:7}
	Suppose $d\geq 6$. Let $\mathbf{X}$ be an isotropic unimodal L\'{e}vy process on $\RR^d$ with the
	L\'{e}vy--Khintchine exponent $\psi$. Then 
	\begin{equation}
		\label{eq:limGeO}
		\lim_{\norm{x} \to \infty} \frac{G(x)}{|x|^{-d} \psi(|x|^{-1})^{-1}} = c > 0,
	\end{equation}
	if and only if $\psi \in \ROrg$ for some $\alpha > 0$. In particular, \eqref{eq:limGeO} implies that
	\[
		c = 2^{-\alpha}\pi^{d/2}\frac{\Gamma\left((d-\alpha)/2\right)}{\Gamma\left(\alpha/2\right)}.
	\]
\end{theorem}
\begin{proof}
	If $\psi\in\ROrg$ for some $\alpha>0$, then in view of \cite[Corollary 3]{cgt} we get \eqref{eq:limGeO}. Therefore,
	it is enough to prove the converse. Again, we are going to use a version of Drasin--Shea--Jordan theorem
	proved in \cite[Theorem 2]{MR1760604}. Since for each $\lambda > 0$, the $\lambda$-resolvent kernel is
	integrable and its Fourier transform equals $(\psi(\xi) + \lambda)^{-1}$, by \eqref{eq:Hankel}, for $r > 0$
	we have
	\begin{equation}
		\label{eq:39}
		\frac{r^{d/2-1}}{\psi(r)+\lambda}
		=
		(2\pi)^{d/2}
		\int^\infty_0 s^{d/2} G^{\lambda}(s) J_{d/2-1}(rs) {\: \rm d}s.
	\end{equation}
	By \eqref{eq:50} and \eqref{eq:limGeO}, there are $s_0 > 0$ and $C > 0$ such that for all $s > s_0$, we can estimate
	\begin{align*}
		s^{d/2} G^\lambda(s) J_{d/2-1}(rs) 
		& \leq C s^{d/2} G(s) \min\big\{(rs)^{d/2 -1}, (rs)^{-1/2}\big\} \\
		& \leq C s^{-d/2} \psi(s^{-1})^{-1} \min\big\{(rs)^{d/2-1}, (rs)^{-1/2}\big\},
	\end{align*}
	which in view of \eqref{eq:31} is integrable on $[s_0, \infty)$. Next, for $0 < s \leq s_0$ we have
	\[
		s^{d/2} G^\lambda(s) J_{d/2 - 1}(r s)
		\leq C G(s) s^{d-1} \min\big\{r^{d/2 - 1}, r^{-1/2} s^{-d/2 + 1/2} \big\},
	\]
	which is integrable on $[0, s_0]$ since $G$ is integrable on $\RR^d$. Hence, the integrand in
	\eqref{eq:39} has integrable majorant independent of $\lambda$. Therefore, by the dominated convergence
	theorem, as $\lambda$ approaches zero we obtain
	\[
		\frac{r^{d/2-1}}{\psi(r)}=(2\pi)^{d/2}\int^\infty_0 s^{d/2}G(s)J_{d/2-1}(rs) {\: \rm d}s.
	\]
	From this point on, the proof follows the same line as in Theorem \ref{thm:6} and is omitted.
\end{proof}

\section{The concentration function}
\label{sec:4}
In this section we study concentration functions $h_1, \ldots, h_d$ defined $r > 0$ by the formula
\[
	h_j(r)=
	\int_{\RR^j} \min\big\{1, r^{-2} \norm{y}^2\big\} \nu_j(y) {\: \rm d} y.
\]
where
\[
	\nu_j(x_1, \ldots, x_j) = \int_{\RR^{d-j}} \nu(x_1, \ldots, x_j, y_1, \ldots, y_{d-j}) {\: \rm d}y_1
	{\rm d}y_{d-j}.
\]
For $j \in \{1, \ldots, d\}$ and $r > 0$ we set
\[
	K_j(r)=r^{-2} \int_{\norm{y} \leq r^2} \norm{y}^2 \nu_j(y) {\: \rm d} y.
\]
We write $h = h_d$.
 
There is a connection between $K_j$, the concentration function $h_j$ and the L\'evy--Khintchine exponent $\psi$,
namely in \cite{MR3350043}, it was shown that for all $r > 0$
\begin{equation}
	\label{eq:46}
	h_j^\prime(r) = - 2 \frac{K_j(r)}{r},
\end{equation}
and
\begin{equation}
	\label{eq:hcomppsi}
	h_j(r) \asymp \psi(1/r).
\end{equation}
Moreover, for all $r > 0$ and $\lambda \geq 1$ we have
\begin{equation}
	\label{eq:55}
	\lambda^2 h_j(\lambda r) \geq h_j(r).
\end{equation}
Indeed,
\begin{align*}
	\lambda^2 h_j(\lambda r) & = \int_{\RR^j} \min\{\lambda^2,r^{-2} \norm{y}^2\} \nu_j(y) {\: \rm d}y \geq \int_{\RR^j} \min\{1,r^{-2} \norm{y}^2\} \nu_j(y) {\: \rm d}y
	 = h_j(r).
\end{align*}
Next, we observe that for all $r > 0$ and $\lambda \geq 1$,
\begin{equation}
	\label{eq:Ksc1}
	\lambda^2 K_j (\lambda r)\geq K_j(r).
\end{equation}
From the other side, for all $r > 0$ and $\lambda \geq 1$, we have
\begin{equation}
	\label{eq:Ksc2}
	\lambda^{-j} K_j(\lambda r) \leq K_j(r),
\end{equation}
because $\nu$ is radially monotonic we can estimate
\begin{align*}
	K_j(\lambda r) & = 
	\lambda^{-2} r^{-2} \int_{\norm{y} \leq \lambda r} \norm{y}^2
	\nu_j(y) \: {\rm d} y\\
	& = 
	\lambda^j r^{-2} \int_{\norm{y} \leq r}
	\norm{y}^2 \nu_j(\lambda y) {\: \rm d} y
	\leq \lambda^j K_j(r).
\end{align*}
The function $K_1$ controls the increments of the L\'evy--Khintchine exponent $\psi$. Indeed, we have
the following lemma.
\begin{lemma}
	\label{K1} 
	Let $\lambda \in [1, 2]$. Then for all $x \in \RR^d$,
	\[
		|\psi(\lambda x)-\psi(x)|
		\leq 
		3 K_1(\norm{x}^{-1}).
	\]
\end{lemma}
\begin{proof} 
	We start by observing that for all $t \geq 0$
	\begin{equation}
		\label{eq:18}
		\bigg|
		\int_0^t
		\cos(u) - \cos(\lambda u)
		{\: \rm d} u
		\bigg|
		=
		\abs{\sin(t) - \lambda^{-1} \sin(\lambda t)} \leq 2 \min\{1, t^3\}.
	\end{equation}
	Let $\mu$ be a measure supported on $(0, \infty)$ defined by
	\[
		\mu((y, \infty)) 
		= \nu_1(y)
		= \int_{\RR^{d-1}} \nu\Big(\sqrt{\abs{x}^2 + y^2}\Big) {\: \rm d} x.
	\]
	Hence, by \eqref{eq:30} and the integration by parts we obtain
	\begin{align*}
		\psi(\lambda u) - \psi(u) &= \int_0^\infty \big(\cos(u y) - \cos(\lambda u y)\big) \nu_1(y) {\: \rm d} y \\
		& = \int_0^\infty \big(\cos(u y) - \cos(\lambda u y)\big) \int_{(y, \infty)} \: \mu({\rm d} x) 
		{\:\rm d} y \\
		& = \int_0^\infty \frac{1}{u} 
		\bigg(\int_0^{u y} \big(\cos(x) - \cos(\lambda x)\big) {\: \rm d} x \bigg) \: \mu({\rm d} y).
	\end{align*}
	Therefore, by \eqref{eq:18},
	\[
		\abs{ \psi(\lambda u) - \psi(u) } \leq 2 \int_0^\infty \frac{1}{u} \min\{1, (u y)^3\} \: \mu({\rm d} y).
	\]
	From the other side, by the change of variables and the Fubini--Tonelli's theorem
	\begin{align*}
		K_1(u^{-1})
		& = 2 \int_0^{u^{-1}} u^2y^2 \nu_1(y) {\: \rm d} y \\
		& = 2 \int_0^{u^{-1}} u^2y^2 \int_{(y,\infty)} \: \mu({\rm d}x) {\: \rm d}y \\ 
		& = 2 \int_0^\infty \bigg( \int_0^{\min\{x, u^{-1}\}} u^2y^2 {\: \rm d}y \bigg) \: \mu({\rm d} x)
		= \frac{2}{3} \int_0^\infty \frac{1}{u} \min\{1, (u x)^3\} \: \mu({\rm d}x). \qedhere
	\end{align*}
\end{proof}
The function $K_d$ naturally appears in estimates of the L\'evy measure. Indeed, there is $C > 0$, depending only on $d$, such that for all
$x \in \RR^d$
\begin{equation}
	\label{eq:Ksc3}
	\nu(x) \le C \norm{x}^{-d} K_d(\norm{x}).
\end{equation}
However, for all $r > 0$
\begin{equation}
	\label{eq:47}
	K_d(r) \leq d K_1(r),
\end{equation}
in general $K_d$ is not comparable to $K_1$. Instead, we have the following lemma.
\begin{lemma}
	\label{lem:3}
	If $d \geq 2$ then there is $C > 0$ such that for all $r > 0$
	\[
		C K_1(r) \leq K_d(r) + r \int_r^\infty s^{d-2} \nu(s) {\: \rm d} s.
	\]
\end{lemma}
\begin{proof}
	We have
	\[
		K_1(r) = \frac{1}{d} K_d(r) + r^{-2} \int_{\atop{\norm{y} \geq r}{\abs{y_1} < r}}
		\abs{y_1}^2 \nu(y) {\: \rm d} y.
	\]
	Next, by taking $y = (u, w) \in \RR \times \RR^{d-1}$ we can write
	\begin{align*}
		r^{-2} \int_{\atop{\norm{y} \geq r}{\abs{y_1} < r}} \abs{y_1}^2 \nu(y) {\: \rm d}y 
		& =
		2 r^{-2} \int_0^r u^2 \int_{\norm{w}^2 \geq r^2 - u^2} \nu(u, w) {\: \rm d} w {\: \rm d} u \\
		& =
		2 r^{-2} \omega_{d-1} \int_0^r u^2 \int_{\sqrt{r^2-u^2}}^\infty \nu(\sqrt{s^2 + u^2}) s^{d-2}
		{\: \rm d} s {\: \rm d}u
	\end{align*}
	where in the last step we have used spherical coordinates in $\RR^{d-1}$. By the change of variables,
	\begin{align*}
		r^{-2} \int_{\atop{\norm{y} \geq r}{\abs{y_1} < r}} \abs{y_1}^2 \nu(y) {\: \rm d}y
		=
		2 r^{-2} \omega_{d-1} \int_0^r u^2 \int_r^\infty (v^2 - u^2)^{(d-3)/2} v \nu(v) {\: \rm d}v {\: \rm d}u.
	\end{align*}
	Since there is $C > 0$ such that for $v \geq r$ 
	\[
		\int_0^r (v^2 - u^2)^{(d-3)/2} u^2 {\: \rm d}u \leq C v^{d-3} r^3,
	\]
	by the Fubini--Tonelli's theorem
	\begin{align*}
		r^{-2} \int_{\atop{\norm{y} \geq r}{\abs{y_1} < r}} \abs{y_1}^2 \nu(y) {\: \rm d}y
		& =
		2 r^{-2} \omega_{d-1} 
		\int_r^\infty
		\bigg(
		\int_0^r u^2 (v^2 - u^2)^{(d-3)/2} {\: \rm d}u \bigg)
		v \nu(v) {\: \rm d}v \\
		& \leq
		C r \int_r^\infty v^{d-2} \nu(v) {\: \rm d}v. \qedhere
	\end{align*}
\end{proof}

\begin{corollary}
	\label{cor:1}
	$K_d$ is bounded if and only if $K_1$ is bounded.
\end{corollary}
\begin{proof}
	If $K_1$ is bounded the conclusion about $K_d$ follows from the estimate \eqref{eq:47}. Conversely,
	if $K_d$ is bounded then by \eqref{eq:Ksc3}, 
	\begin{align*}
		\int_r^\infty u^{d-2}\nu(u) {\: \rm d} u
		& \leq
		C\sup_{s > 0} K_d(s) \int_r^\infty u^{-2} {\: \rm d} u \\
		& =
		C r^{-1} \sup_{s > 0} K_d(s).
	\end{align*}
	Now, applying Lemma \ref{lem:3} we can easily deduce that $K_1$ must be bounded.
\end{proof}

\begin{corollary}
	\label{SP1}
	Suppose that there are $C > 0$ and $\alpha < 1$ such that for all $\lambda \geq 1$ and $r > 0$
	\begin{equation}
		\label{eq:51}
		\nu(\lambda r) \leq C \lambda^{-d+\alpha} \nu(r).
	\end{equation}
	Then
	\[
		K_1(r) \asymp K_d(r).
	\]
\end{corollary}
\begin{proof}
	For the proof, let us observe that, by \eqref{eq:51}, for $s \geq r$
	\[
		\nu(s) \leq C s^{-d+\alpha} r^{d - \alpha} \nu(r),
	\]
	thus
	\[
		r \int_r^{\infty} s^{d-2} \nu(s) {\: \rm d}s
		\le 
		C r^{d-\alpha+1} \nu(r) \int_r^{\infty} s^{\alpha-2} {\: \rm d} s 
	\]
	which, by \eqref{eq:Ksc3}, is bounded by a constant multiply of $K_d(r)$.
\end{proof}

\begin{remark}
	Corollary \ref{SP1} can be applied for self-decomposable processes. Indeed, by \cite[Theorem 15:10]{MR1739520}
	there is a non-increasing function $g: (0, \infty) \rightarrow [0, \infty)$ such that
	$\nu(x) = \norm{x}^{-d} g(\norm{x})$. Therefore, for $\lambda \geq 1$ and $r > 0$
	\[
		\nu(\lambda r) = \lambda^{-d} \frac{g(\lambda r)}{g(r)} \nu(r)
		\leq
		\lambda^{-d} \nu(r).
	\]
\end{remark}

\begin{proposition}\label{slow10}
	If 
	\[
		\lim_{r\to 0^+} \frac {K_d(r)}{h(r)}=0
	\]
	then $\psi \in \calR^\infty_0$.
\end{proposition}
\begin{proof}
	Let us notice that by \eqref{eq:hcomppsi} and Lemma \ref{K1}, for $\lambda \in [1, 2]$ we have
	\[
		\bigg|\frac{\psi(\lambda x)}{\psi(x)} - 1\bigg|
		\leq
		C \frac{K_1(\norm{x}^{-1})}{h(\norm{x}^{-1})},
	\]
	thus, in view of Lemma \ref{lem:3}, we need to show that
	\[
		\limsup_{r \to 0^+} \frac{r}{h(r)} \int_r^{\infty} s^{d-2} \nu(s) {\: \rm d} s = 0.
	\]
	Given $\epsilon > 0$, let $\delta > 0$ be such that for $r \leq \delta$
	\[
		\frac{K_d(r)}{h(r)} \leq \epsilon.
	\]
	By \eqref{eq:Ksc3} and the monotonicity of $h$ we have
	\begin{align*}
		\int_r^{\infty} s^{d-2}\nu(s) {\: \rm d} s
		& = 
		\int_r^\delta s^{d-2} \nu(s) {\: \rm d} s
		+ \int_\delta^{\infty} s^{d-2}\nu(s) {\: \rm d}s\\
		& \le C \int_{r}^\delta s^{-2} \frac{K_d(s)}{h(s)} h(s) {\: \rm d}s
		+ \int_\delta^{\infty} s^{d-2}\nu(s) {\: \rm d}s \\
		&\le C \epsilon \cdot h(r) \cdot \int_r^\delta s^{-2} {\: \rm d} s
		+ \int_\delta^{\infty} s^{d-2}\nu(s) {\: \rm d} s.
	\end{align*}
	Again, by monotonicity of $h$ we have
	\[
		\lim_{r \to 0^+} \frac{r}{h(r)} = 0,
	\]
	thus
	\[
		\limsup_{r \to 0^+} \frac{r}{h(r)} \int_r^{\infty} s^{d-2} \nu(s) {\: \rm d} s 
		\leq
		C \epsilon. \qedhere
	\]
\end{proof}

\begin{proposition}
	\label{nslow10}
	There are $C, \alpha > 0$ and $\theta \geq 0$ such that $\psi \in \WLSC{\alpha}{\theta}{C}$ if and only if
	\begin{equation}
		\label{eq:52}
		\liminf_{r \to 0^+} \frac{K_d(r)}{h(r)} > 0.
	\end{equation}
	In particular, \eqref{eq:52} implies that $\psi \notin \calR_0^\infty$.
\end{proposition}
\begin{proof}
	If \eqref{eq:52} is satisfied, then there are $C, \theta > 0$ such that for $r \in (0, \theta)$
	\begin{equation}
		\label{eq:53}
		K_d(r) \geq C h(r).
	\end{equation}
	Therefore, for $r \in (0, \theta)$
	\[
		\big(r^{2C} h(r)\big)^\prime = 2 C r^{2C - 1} h(r) + r^{2C} h^\prime(r) \leq 0,
	\]
	which implies that the function $(\theta^{-1}, \infty) \ni r \mapsto r^{-2C} h(r^{-1})$ is non-decreasing. Hence,
	$h(r^{-1})$ belongs to $\WLSC{2C}{\theta^{-1}}{1}$. Finally, by \eqref{eq:hcomppsi}, $\psi$ belongs to
	$\WLSC{2C}{\theta^{-1}}{c}$ for some $c \in (0, 1]$.

	Conversely, if $\psi\in \WLSC{\alpha}{\theta}{c}$ then $h(r^{-1})$ belongs to $\WLSC{\alpha}{\theta}{c'}$ for some
	$c' \in (0, 1]$. In particular, there is $\lambda \in (0, 1)$ such that for all $r \in (0, \theta^{-1})$
	\[
		h(r) \leq h(\lambda r) - h(r).
	\]
	Since $r^2 K_d(r)$ is non-decreasing, by \eqref{eq:46} we have
	\begin{align*}
		h(r) \leq h(\lambda r) - h(r) & = 2 \int_{\lambda r}^r \frac{K_d(s)}{s} {\: \rm d} s \\
		& \leq 2 r^2 K_d(r) \int_{\lambda r}^r s^{-3} {\: \rm d}s = (\lambda^{-2} - 1) K_d(r),
	\end{align*}
	which finishes the proof.
\end{proof}

\section{Estimates for heat kernels and Green functions}
\label{sec:5}
In this section we prove estimates of the transition density $p(t,x)$ and the Green potentials $G^\lambda(x)$ under
the assumption that the process $\mathbf{X}$ is isotropic and unimodal. The main result is Theorem \ref{thm:4}.
The special case, subordinate Brownian motion, is considered in Section \ref{sec:5.1}. For $r>0$,  by  $B_r$  we denote the ball in $\RR^d$ of radius $r$ centred at the origin. 

\subsection{Estimates from below}
The key tool in proving lower bounds is the following lemma proved in \cite{BGR3}.
\begin{lemma}[{\cite[Lemma 1.11]{BGR3}}]
	\label{HKLB}
	Let $\mathbf{X}$ be an isotopic unimodal L\'evy process on $\RR^d$. Then for all $t > 0$
	and $x \in \RR^d$
	\[
		p(t, x) \geq 4^{-d} t \nu(x) \left(\PP(\tau_{B_{\norm{x}/2}}>t)\right)^2 
	\]
	where for a Borel set $A$ we have set
	\[
		\tau_A = \inf\big\{t > 0 : X_t \notin A \big\}.
	\]
	\end{lemma}
The second result essential for our argument is an improvement of \cite[Theoreme 2.1]{dupuis}
where one-dimensional symmetric L\'{e}vy processes were considered. For a related estimates in the case of
one-dimensional Feller processes see \cite[Lemma 7]{MR3188355}.
\begin{proposition}
	\label{tau}
	There exist $c_1, c_2 > 0$ such that for any isotropic L\'{e}vy process
	\[
		c^{-1}_1 e^{-c^{-1}_2 \,t\, h(r)} \leq \PP(\tau_{B_r} > t) \leq c_1 e^{-c_2 \,t\, h(r)}
	\]
	for all $t, r > 0$.
\end{proposition}
\begin{proof}
	For $s, t \geq 0$ we set
	\[
		M_{s,t}=\sup_{s\leq u\leq t} \abs{X_u-X_s},
	\]
	and $M_t=M_{0,t}$. By \cite[(3.2)]{Pruitt}, there exists a constant $C > 0$ depending only on the dimension
	$d$ such that for any $r, t > 0$
	\begin{equation}
		\label{eq:108}
		\PP(M_t \geq r) \leq C t h(r),
		\qquad \PP(M_t\leq r)\leq \frac{C}{t h(r)}.
	\end{equation}
	In particular, by taking $t_0 = \frac{8C}{h(r)}$, in view of \eqref{eq:55} we obtain
	\[
		\PP(\tau_{B_{2r}} > t_0) = \PP(M_{t_0} < 2r) \leq \frac{1}{2}.
	\]
	Therefore, by the Markov property, for $n \in \NN$,
	\begin{align*}
		\PP(\tau_{B_r} > (n+1) t_0) &= \EE \left( \tau_{B_r}> nt_0 ;\,\PP^{X_{nt_0}}(\tau_{B_r}>t_0)\right) \\
		&\leq \EE\left(\tau_{B_r}> nt_0;\,\PP^{X_{nt_0}}(\tau_{B(X_{nt_0},2r)}>t_0)\right)
		\leq \frac{1}{2}\PP(\tau_{B_r}> nt_0).
	\end{align*}
	Hence, for all $n \in \NN$
	\[
		\PP(\tau_{B_r}>nt_0) \leq 2^{-n}.
	\]
	Given $t > 0$, let $n = \lfloor t / t_0 \rfloor$. Then
	\[
		\PP(\tau_{B_r}>t) \leq \PP(\tau_{B_r} > n t_0) \leq 2^{1-t/t_0}=2^{1-t h(r)/(8C)}.
	\]

	To prove the lower bound we consider, for $T>0$ (to be specified later),  
$$A^T_k= \{M_{kT,(k+1)T}< r,\,2\sprod{X_{(k+1)T}-X_{kT}}{X_{kT}}\leq -|X_{(k+1)T}-X_{kT}||X_{kT}|\}.$$
Observe that for any $a,b\in[0,1]$ $$a^2+b^2-ab\leq 1.$$
Hence for  $|x|,|y|\leq  r$ if $2\sprod{x}{y}\leq -|x||y|$ we have $|x+y|\leq r$ as well. 
This implies, for any $n=0,1,2,\ldots$,
$$  \{M_{nT}< 2r\}\supset \bigcap^{n-1}_{k=0}A_k^T$$
and $|X_{nT}|\leq r$. 
Next, by the Markov property we have, for $n=0,1,\ldots$,
\begin{align*}\PP(\tau_{B_{2r}}>(n+1)T)&=\PP(M_{(n+1)T}< 2r) \geq \PP(\bigcap^{n}_{k=0}A_k^T)=\EE\left( \bigcap^{n-1}_{k=0}A_k^T;\PP^{X_{(n-1)T}}(A_n^T)\right)\\&\geq \inf_{|z|\leq r}\PP^z(A_0^T)\EE\left( \bigcap^{n-1}_{k=0}A_k^T\right)\geq \left(\inf_{|z|\leq r}\PP^z(A_0^T)\right)^{n+1}.\end{align*}
Since $X_t$ is isotropic $\inf_{|z|\leq r}\PP^z(A_0^T)= \PP^{z_0}(A_0^T)$ for any $z_0$: $|z_0|=r$.
Therefore $$\inf_{|z|\leq r}\PP^z(A_0^T)\geq \PP^0(M_T< r)-\PP^0(2\sprod{X_T}{z_0}\geq -|X_T|r).$$
Again, by isotropicity of $X_T$ we have $\PP^0(2\sprod{X_T}{z_0}\geq -|X_T|r)=C_1(d)<1$ for any $T>0$ and by \eqref{eq:108}
$\PP(M_T< r)=1-\PP(M_T\geq r)\geq (1+C_1)/2$ for $T= (1-C_1)/(2C h(r))$. This ends the proof of  the lower bound by \eqref{eq:55}.

\end{proof}

An immediate corollary to Lemma \ref{HKLB} and Proposition \ref{tau} is the following proposition.

\begin{proposition}
	\label{HKLB1}
	There exist constants $C, c > 0$ such that for any isotropic unimodal L\'{e}vy process $\mathbf{X}$
	and for all $t>0$, $\lambda \ge 0$ and $x \in \Rd$ we have 
	\begin{equation}
		\label{eq:57}
		p(t,x) \ge C t \nu(x) e^{-c t\psi(1/|x|)}
	\end{equation}
	and
	\[
		G^\lambda(x)\ge C \nu(x) \big(\lambda+\psi(1/|x|)\big)^{-2}.
	\]
	The constants $C, c$ depends only on the dimension $d$.
\end{proposition}

\begin{proof}
	Since the constant in \eqref{eq:hcomppsi} depends only on $d$, the first estimate follows from Lemma \ref{HKLB}
	and Proposition \ref{tau}. Now, using \eqref{eq:57} we can estimate
	\begin{align*}
		G^\lambda(x) = \int_0^\infty e^{-\lambda t} p(t, x) {\: \rm d}t
		& \geq
		C \nu(x) \int_0^\infty t e^{-t(\lambda + c \psi(1 / \norm{x}))} {\: \rm d}t\\
		& \geq
		C' \nu(x) \big(\lambda + \psi(1/\norm{x})\big)^{-2}.
		\qedhere
	\end{align*}
\end{proof}

\subsection{Estimates from above}
Our method for obtaining upper bounds is based on the following elementary observation: there is a positive constant $C$
depending only on the dimension $d$, such that for all $t, r > 0$
\begin{equation}
	\label{basic}
	r^{d} p(t, r)
	\leq
	C 
	\PP\Big(\frac{r}{2}\leq \abs{X_t} < r\Big),
\end{equation}
As the first consequence, we obtain the following theorem.
\begin{theorem}
	\label{GUB3}
	Let $\mathbf{X}$ be an isotropic unimodal L\'{e}vy process on $\RR^d$. Then there is $C > 0$ such that for all
	$t > 0$ and $x \in \RR^d \setminus \{0\}$
	\[
		p(t, x) \leq C t \norm{x}^{-d} K_d(|x|).
	\]
	The constant $C$ depends only on the dimension $d$.
\end{theorem}
\begin{proof}
		Let $\calA$ be an infinitesimal generator
	of $\mathbf{X}$, that is
	\[
		\calA f(y) = \lim_{t \downarrow 0} \frac{\EE f(X_t+y) - f(y)}{t}, \quad y\in \Rd. 
	\]Now we consider a function $f \in \calC^2_c(\Rd)$ such that $0\leq f\leq 1$ and  
	$f(x)= 1$, for $1/2\leq |x|\leq 1$ and $f(x)=0$ for $|x|\leq 1/4$ or $|x|>5/4$. Then for $f_r(x)=f(x/r)$ one can show using Taylor expansion for $f_r$, the representation of generator (see \cite[Theorem 31.5]{MR1739520}) and monotonicity of $\nu$ (compare with \cite[the proof of Lemma 3]{harnack}) that
$$	\sup_{y\in\Rd} \mathcal{A} f(y)\leq c(d)(||f''||_{\infty}K(r)+r^d\nu(r)).$$
	 Since $r^d\nu(r)\leq c(d)K(r)$ we have
	\[
		\sup_{y\in\Rd} \mathcal{A} f(y) \leq C K_d(r),
	\]
	for a positive constant $C$ depending only on the dimension $d$. 
	
	Therefore, by the Dynkin's formula we get
	\[
		\PP\Big( \frac{r}{2} \leq \abs{X_t} < r \Big)
		\leq \EE f(X_t)
		=\EE\int^t_0\mathcal{A}f(X_s){\: \rm d}s
		\leq C t K_d(r).
	\]
	Hence, in view of \eqref{basic} we conclude the proof.
\end{proof}

\begin{lemma}
	\label{lem:pt_Gen1}
	Suppose $\omega: (0, \infty) \rightarrow (0, \infty)$ satisfies
	\[
		\abs{\psi(2 r) - \psi(r)} \leq \omega(r),
	\]
	for all $r > 0$. Then there is $C > 0$ such that for all $t > 0$ and $x \in \RR^d \setminus \{0\}$
	\[
		p(t, x)
		\leq
		C t \norm{x}^{-d}
		\int_{\Rd} \omega\big(\norm{y}/\norm{x}\big) e^{-\norm{y}^2/ 4} {\: \rm d} y.
	\]
	The constant $C$ depends only on the dimension $d$.
\end{lemma}
\begin{proof}
	Monotonicity of $p(t, \:\cdot\:)$ together with \eqref{basic} imply
	\begin{align*}
		r^{d/2+1}p(t, \sqrt{r}) & \leq \frac{d+2}{ 2}\int^{r}_0 s^{d/2}p(t, \sqrt{s}) {\: \rm d}s \\
		& \leq \frac{e(d+2)}{2 } \int^{\infty}_0 e^{-s/r}s^{d/2}p(t,\sqrt{s})	{\: \rm d}s \\
		& \leq \frac{Ce(d+2)}{2} \int^{\infty}_0 e^{-s/r} \PP\Big(\frac{\sqrt{s}}{2}\leq \abs{X_t}< \sqrt{s}\Big)
		{\: \rm d}s.
	\end{align*}
	Taking $\lambda = 1/r$ in \eqref{eq:26}, we obtain
	\[
		r^{-1} \int^{\infty}_0 e^{-s/r} \PP\Big(\frac{\sqrt{s}}{2} \leq |X_t|< \sqrt{s}\Big){\: \rm d} s
		= (4\pi)^{-d/2}
		\int_{\Rd}\Big(e^{-t\psi(\norm{x}/\sqrt{r})}-e^{-t\psi(2\norm{x} /\sqrt{r})}\Big)e^{-|x|^2/4} {\: \rm d}x
	\]
	Hence, there is $C > 0$ depending on $d$ such that for any $t, r>0$
	\begin{equation}
		\label{eq:p_t10}
		r^d p(t, r) \leq 
		C \int_{\Rd} \Big(e^{-t \psi(\norm{x} /r)}-e^{-t \psi(2 \norm{x} / r)} \Big) e^{-\norm{x}^2/4} {\: \rm d} x.
	\end{equation}
	Since for any $a, b \geq 0$, $e^{-a}-e^{-b} \leq \abs{b-a}$, we obtain
	\begin{align*}
		r^{d} p(t, r) 
		&\leq C t \int_{\Rd} \big|\psi(\norm{x}/r) - \psi(2 \norm{x}/r) \big| e^{-\norm{x}^2/4} {\: \rm d} x\\
		&\leq C t \int_{\Rd} \omega(\norm{x}/r) e^{-\norm{x}^2/4} {\: \rm d}x,
	\end{align*}
	which completes the proof.
\end{proof}
\begin{remark}
	To obtain a more precise upper bound we impose on the function $\omega$ a condition that there are $c > 0$, $\alpha < d$
	and $\beta > 0$ such that for all $\lambda, x > 0$
	\begin{equation}
		\label{eq:58}
		\omega(\lambda x) \leq c \omega(x) \max\{\lambda^{-\alpha}, \lambda^\beta\}.
	\end{equation}
	Then there is $ C> 0$ such that for all $t > 0$ and $x \in \RR^d \setminus\{0\}$,
	\begin{equation}
		\label{eq:60}
		p(t, x) \leq C t \norm{x}^{-d} \omega(1/ \norm{x}).
	\end{equation}
	Indeed, by Lemma \ref{lem:pt_Gen1},
	\begin{align*}
		p(t, x) & \leq C t \norm{x}^{-d} \int_{\RR^d} \omega\big(\norm{y} / \norm{x}\big) e^{-\norm{y}^2/4} {\: \rm d}y \\
		& \leq
		C t \norm{x}^{-d} \omega(1/ \norm{x}) 
		\int_{\RR^d} \max\big\{\norm{y}^{-\alpha}, \norm{y}^\beta\big\}
		e^{-\norm{y}^2/4} {\: \rm d}y.
	\end{align*}
	The constant in \eqref{eq:60} depends on $c$, $\alpha$, $\beta$ and the dimension $d$.
\end{remark}

\begin{proposition}
	\label{GUB}
	Suppose that
	\[
		\sup_{0 < x \leq y \leq 2x}
		\big| \psi(x) - \psi(y) \big| < \infty.
	\]
	Assume that there is $\omega: (0, \infty) \rightarrow (0, + \infty)$ satisfying
	\begin{equation}
		\label{eq:62}
		\big|\psi(2 r) - \psi(r) \big|\leq \omega(r),
	\end{equation}
	for all $r > 0$ such that there are $C > 0$, $\alpha < d$ and $\beta > 0$ with
	\begin{equation}
		\label{eq:61}
		\omega(\lambda x) \leq c \omega(x) \max\{\lambda^{-\alpha}, \lambda^{\beta}\}
	\end{equation}
	for all $\lambda, x > 0$. Then there are $C > 0$ and $t_0 > 0$ such that for all $t \in (0, t_0)$ and
	$x \in \RR^d \setminus \{0\}$
	\[
		p(t, x) \leq C t \norm{x}^{-d} \omega(1/\norm{x}) e^{-t \psi(1/\norm{x})}.
	\]
	The constant $C$ depends on $c$, $\alpha$, $\beta$ and the dimension $d$.
\end{proposition}
\begin{proof}
	Thanks to \eqref{eq:62}, we can apply Lemma \ref{lem:1}, thus, there are $C_1 > 0$ and $\delta > 0$ such that
	for all $x, y \in \RR^d \setminus \{0\}$,
	\[
		e^{t \phi(1/\norm{x})} \Big|e^{-t \psi(\norm{y}/\norm{x})} - e^{-t\psi(2 \norm{y}/\norm{x})} \Big|
		\leq
		t C_1^t \max\big\{\norm{y}, \norm{y}^{-1} \big\}^{\delta t} \omega(\norm{y}/\norm{x}),
	\]
	which together with \eqref{eq:61} gives
	\[
		e^{t \phi(1/\norm{x})} \Big|e^{-t \psi(\norm{y}/\norm{x})} - e^{-t\psi(2 \norm{y}/\norm{x})} \Big|
		\leq
		c t C_1^t \max\big\{\norm{y}^{-\alpha + \delta t}, \norm{y}^{\beta + \delta t}\big\}
		\omega(1/\norm{x}).
	\]
	Now, by \eqref{eq:p_t10}, 
	\[
		p(t, x) \leq  c C t C_1^t \norm{x}^{-d} \omega(1/\norm{x}) 
		\int_{\RR^d} 
		\max\big\{\norm{y}^{-\alpha + \delta t}, \norm{y}^{\beta + \delta t}\big\}
		e^{-\norm{y}^2/4} {\: \rm d} y.
	\]
	Taking $t_0 = (d+\alpha)/\delta$, the last integral is finite for $t \in (0, t_0)$.
\end{proof}
In view of Lemma \ref{K1} the natural candidate for $\omega$ is the function $K_1$. 
\begin{remark}
	\label{rem2}
	Suppose that $d\ge 2$. If $K_d$ is a bounded function then there are $C > 0$ and $t_0 > 0$ such that
	for all $t \in (0, t_0)$ and $x \in \RR^d \setminus \{0\}$
	\[
		p(t, x) \leq C t \norm{x}^{-d} K_1(|x|) e^{-t\psi(1/|x|)}.
	\]
	The constant $C$ depends on the dimension $d$. Indeed, by \eqref{eq:Ksc1} and \eqref{eq:Ksc2}, the function
	$K_1$ belongs to $\WLSC{-1}{0}{1} \cap \WUSC{2}{0}{1}$. Hence, by Lemma \ref{K1} and Corollary \ref{cor:1}
	the claim follows from Proposition \ref{GUB}.

	For $d = 1$, we need to impose better scaling properties of $K_1$ to be able to apply Proposition \ref{GUB}.
\end{remark}

\subsection{Sharp two-sided estimates}
The main theorem of this section is Theorem \ref{thm:4}.
\begin{theorem}
	\label{thm:4}
	Let $\mathbf{X}$ be an isotropic unimodal L\'{e}vy process on $\RR^d$ with the L\'{e}vy--Khintchine
	exponent $\psi \in \Pi^\infty_\ell$ for some bounded $\ell \in \calR_0^\infty$. Then there are $r_0, t_0 > 0$
	such that, for all $t \in (0, t_0)$ and $0<|x|\le r_0$,
	\begin{equation}
		\label{ULB}
		{p(t, x)} \asymp {t \norm{x}^{-d} \ell(\norm{x}^{-1}) e^{-t \psi(\norm{x}^{-1})}}.
	\end{equation}
	The implicit constants, $r_0$ and $t_0$ depend on the dimension $d$ and the process $\mathbf{X}$.
\end{theorem}
\begin{proof}
	Since $\psi \in \Pi^\infty_\ell$, by Theorem \ref{thm:5} there is $r_0 > 0$ such that
	\[
		\nu(x) \asymp \norm{x}^{-d} \ell(1/\norm{x})
	\]
	for all $\norm{x} \leq r_0$. Hence, by Corollary \ref{HKLB1}, the lower estimate for $p(t, x)$ holds whenever
	$t \psi(1/|x|) \le 1$. If $t\psi(1/|x|)\ge 1$ and $t$ is sufficiently small, we apply Theorem \ref{thm:3}.

	The upper estimate follows from Proposition \ref{GUB} because by \eqref{eq:5} we have
	\[
		\big|\psi(2x) - \psi(x) \big| \asymp \ell(\norm{x}).
		\qedhere
	\]
\end{proof}

\begin{theorem}
	\label{Gbound}
	Suppose $d \geq 6$. Let $\mathbf{X}$ be an isotropic unimodal L\'evy process on $\RR^d$.
	Then there is $C > 0$ such that for all $x \in \RR^d \setminus \{0\}$ 
	\[
		G(x) \leq C \norm{x}^{-d} h(\norm{x})^{-2} K_1(\norm{x}).
	\]
	The constant $C$ depends only on the dimension $d$.
\end{theorem}
\begin{proof}
	Integrating both sides in \eqref{eq:p_t10} with respect to $t \in (0,\infty)$ we get
	\[
		r^d G(r) \leq C \int_0^\infty
		\left( \frac{1}{\psi(s/r)} - \frac{1}{\psi(2 s/r)} \right)
		s^{d-1} e^{-s^2/4}
		{\: \rm d}s.
	\]
	Hence, by Lemma \ref{K1} and \eqref{eq:hcomppsi},
	\[
		r^d G(r) \leq C C' \int_0^\infty h(r/s)^{-2} K_1(r/s) s^{d-1} e^{-s^2/4} {\: \rm d}s.
	\]
	Now, let us observe that by \eqref{eq:Ksc1} and monotonicity of $h$, the function
	\[
		u \mapsto u^2 h(u)^{-2} K_1(u)
	\]
	is non-decreasing, whereas by \eqref{eq:Ksc2} and monotonicity of $r^2 h(r)$, the function
	\[
		u \mapsto u^{-5} h(u)^{-2} K_1(u)
	\]
	is non-increasing. Therefore,
	\[
		\int_0^1 h(r/s)^{-2} K_1(r/s) s^{d-1} e^{-s^2/4} {\: \rm d}s
		\leq
		h(r)^{-2} K_1(r) \int_0^1 s^{d+1} e^{-s^2/4} {\: \rm d}s,
	\]
	and
	\[
		\int_1^\infty h(r/s)^{-2} K_1(r/s) s^{d-1} e^{-s^2/4} {\: \rm d}s
		\leq
		h(r)^{-2} K_1(r) \int_1^\infty s^{d-6} e^{-s^2/4} {\: \rm d}s,
	\]
	which concludes the proof of the theorem.
\end{proof}

\begin{corollary}
	\label{cor:3}
	Suppose $d \ge 6$. Assume there are $-2< \beta \leq \alpha<1$ and $C_1, C_2 > 0$ such that
	for all $\lambda \geq 1$ and $r > 0$
	\begin{equation}
		\label{eq:63}
		C_1 \lambda^{-d + \beta} \nu(r) \leq \nu(\lambda r) \leq C_2 \lambda^{-d + \alpha} \nu(r).
	\end{equation}
	Then
	\[
		G(x) \asymp \norm{x}^{-d} h(\norm{x})^{-2} K_d(r).
	\]
\end{corollary}
\begin{proof}
	By the first inequality in \eqref{eq:63} and Lemma \ref{lem:3}, there is a constant $C > 0$ such that
	$K_1 \leq C K_d$, thus the upper estimate follows from Theorem \ref{Gbound}. The second inequality implies
	that $\nu(r) \geq C r^{-d} K_d(r)$, thus the lower estimate is the consequence of Corollary \ref{HKLB1}.
\end{proof}

We conjecture that the upper estimate in Corollary \ref{cor:3} is true for any unimodal isotropic process provided 
the dimension $d \ge 3$. 

A consequence of Theorem \ref{Gbound} is Theorem \ref{thm:equiG}, which generalizes \cite[Theorem 5]{harnack}
to isotropic unimodal L\'{e}vy processes, provided the dimension $d \geq 6$.
\begin{theorem}
	\label{thm:equiG}
	Suppose $d \geq 6$. Let $\mathbf{X}$ be an isotropic unimodal L\'{e}vy process on $\RR^d$ with the L\'evy--Khintchine
	exponent $\psi$. Then there exist $C, R > 0$ such that for all $\norm{x} \leq R$
	\begin{equation}
		\label{eq:64a}
		G(x) \geq C \norm{x}^{-d} h(\norm{x})^{-1}
	\end{equation}
	if and only if there are $R, c, \alpha > 0$ such that $\psi$ satisfies $\WLSC{\alpha}{R^{-1}}{c}$.
\end{theorem}
\begin{proof}
	If $\psi$ satisfies $\WLSC{\alpha}{R}{c}$, for some $R, \alpha, c > 0$, then the estimate for the Green function
	$G$ follows by \cite[Theorem 3]{harnack}. Conversely, if we assume \eqref{eq:64a} then by Theorem \ref{Gbound} there is
	$c > 0$ such that for all $\norm{x} \leq R$
	\[
		h(\norm{x})^{-1} \leq c h(\norm{x})^{-2} K_1(\norm{x}).
	\]
	Then, by \eqref{eq:hcomppsi}, there is $c > 0$ such that for all $\norm{x} \leq R$
	\[
		h_1(\norm{x})^{-1} \leq c h_1(\norm{x})^{-2} K_1(\norm{x}).
	\]
	Hence, by Proposition \ref{nslow10} applied to the one-dimensional projection of $\mathbf{X}$ we obtain the claim.
\end{proof}
Similarly one can prove the following result.\begin{theorem}
	\label{thm:equiG1}
	Suppose $d \geq 6$. Let $\mathbf{X}$ be an isotropic unimodal L\'{e}vy process on $\RR^d$ with the L\'evy--Khintchine
	exponent $\psi$. Then there exist $C, R > 0$ such that for all $\norm{x} \geq R$
	\begin{equation}
		\label{eq:64}
		G(x) \geq C \norm{x}^{-d} h(\norm{x})^{-1}
	\end{equation}
	if and only if there are $R, c, \alpha > 0$ such that $\psi(1/r)$ satisfies $\WLSC{\alpha}{R^{-1}}{c}$.
\end{theorem}
\subsection{Subordinate Brownian Motions}
\label{sec:5.1}
In this section we consider a pure-jump subordinate Brownian motion $\mathbf{X} = (X_t : t \geq 0)$ with the
L\'evy--Khintchine exponent $\psi(x) = \varphi(\norm{x}^2)$ where $\varphi: [0, \infty) \rightarrow [0, \infty)$
is a Bernstein function such that $\varphi(0)=0$. Let $\mathbf{B} = (B_t : t \geq 0)$ be a Brownian motion on $\RR^d$.
The process $\mathbf{X}$ can be constructed as the time changed $\mathbf{B}$ by an independent subordinator
$(T_t : t \geq 0)$ with the Laplace exponent $\varphi$, that is for $t \geq 0$
\[
	X_t= B_{T_t}.
\]
Let us recall that there is a measure $\mu$ supported on $[0, \infty)$ such that for $u \geq 0$
\begin{equation}
	\label{eq:65}
	\varphi(u)=\int_{[0, \infty)} (1-e^{-us}) \: \mu({\rm d}s),
\end{equation}
and
\[
	\int_{[0, \infty)} \min\{1, s\} \: \mu({\rm d} s).
\]
The measure $\mu$ is the L\'evy measure of the subordinator. If $\mu$ is absolutely continuous then we will denote
its density by $\mu$ as well.
\begin{lemma}
	\label{K_SBM}
	For all $r > 0$
	\[
		\int^{r^2}_0u^{d/2} \varphi'(u) {\: \rm d} u \asymp r^d K_d(r^{-1}).
	\]
	The comparability constant depends only on the dimension $d$.
\end{lemma} 
\begin{proof}
	We observe that by taking derivative of \eqref{eq:65} and Fubini--Tonelli's theorem
	\begin{align*}
		\int^{r^2}_0 u^{d/2} \varphi'(u) {\: \rm d} u 
		& = \int^{r^2}_0 u^{d/2} \int^\infty_0 s e^{-us} \: \mu({\rm d}s) {\: \rm d} u \\
		& = \int^\infty_0 s^{-d/2} \int^{s r^2}_0 u^{d/2} e^{-u} {\: \rm d} u \: \mu({\rm d}s).
	\end{align*}
	Since
	\begin{equation}
		\label{eq:66}
		\int_0^{r} 
		u^{d/2} e^{-u} {\: \rm d} u
		\asymp 
		\min\big\{1, r\big\}^{d/2 + 1},
	\end{equation}
	we get
	\[
		\int_0^{r^2} u^{d/2} \varphi'(u) {\: \rm d} u
		\asymp
		\int_0^\infty s \min\big\{s^{-d/2 - 1}, r^{d + 2} \big\} \: \mu({\rm d} s).
	\]
	From the other side, for all $r > 0$
	\begin{equation}
		\label{eq:75}
		\nu(u) = \int_0^\infty (4 \pi s)^{-d/2} e^{-u^2/(4s)} \mu({\: \rm d}s),
	\end{equation}
	thus by Fubini--Tonelli's theorem
	\begin{align*}
		r^2 K_d(1/r) & = \omega_d \int_0^{1/r} u^{d+1} \nu(u) {\: \rm d}u \\
		& = \omega_d \int_0^\infty (4 \pi s)^{-d/2} \int_0^{1/r} e^{-u^2/(4s)} u^{d+1}
		{\: \rm d}u \: \mu({\rm d} s) \\
		& = 2 \omega_d \pi^{-d/2} \int_0^\infty s \int_0^{1/(4 s r^2)} e^{-u} u^{d/2}
		{\: \rm d}u \: \mu({\rm d} s).
	\end{align*}
	Finally, by \eqref{eq:66}
	\[
		r^d K_d(1/r) \asymp \int_0^\infty s \min\big\{r^{d+2}, s^{-d/2-1}\big\} \: \mu({\rm d} s),
	\]
	which finishes the proof.
\end{proof}

\begin{corollary}
	There is a constant $C > 0$ such that for all $t > 0$ and $x \in \RR^d \setminus \{0\}$
	\[
		p(t, x) \leq C t \int^{|x|^{-2}}_0 u^{d/2} \varphi'(u) {\: \rm d}u.
	\]
	The constant $C$ depends only on the dimension $d$.
\end{corollary}
\begin{proof}
	This is a consequence of Theorem \ref{GUB3} and Lemma \ref{K_SBM}.
\end{proof}
\begin{proposition}
	\label{ubound_p_t}
	Suppose that there exists $\beta \in [0, d/2+1)$ such that $\varphi' \in \WLSC{-\beta}{\theta}{\underline{c}}$, 
	then there is $C > 0$, dependent on $\theta$, $\beta$ and $d$, such that for all $t > 0$ and $\norm{x} \leq \theta^{-1/2}$
	\[
		p(t, x) \leq C \underline{c}^{-1} t \norm{x}^{-d-2} \varphi'(\norm{x}^{-2}).
	\]
\end{proposition}
\begin{proof}
	By the scaling property, for all  $\theta\le u\le r^{-2}$,
	\[
		\varphi'(u)  \le 
	\underline{c}^{-1} (ur^{2})^{-\beta} \varphi'(r^{-2}).
	\]
	Hence,
	\begin{align*}
		\int^{r^{-2}}_0 u^{d/2} \varphi'(u) {\: \rm d} u
		& = \int^{\theta}_0 u^{d/2} \varphi'(u) {\: \rm d} u +\int^{r^{-2}}_\theta u^{d/2} \varphi'(u) {\: \rm d} u\\
		& \le  \theta^{d/2} \varphi(\theta) +
		\underline{c}^{-1}
		r^{-2\beta} \varphi'(r^{-2}) \int^{r^{-2}}_0 u^{d/2-\beta} {\: \rm d} u \\
		&\le \underline{c}^{-1}\left(\frac {\varphi(\theta)}  {\theta\varphi'(\theta)} +
		(d/2 - \beta + 1)^{-1}\right)
		r^{-d - 2} \varphi'(r^{-2}). \qedhere
	\end{align*}

\end{proof}

\begin{remark}
	If the L\'{e}vy measure of the subordinator has a non-increasing density and $d \geq 3$ then
	there is $C > 0$ such that for all $t > 0$ and $x \in \RR^d$
	\begin{equation}
		\label{eq:67}
		p(t, x) \leq C t \norm{x}^{-d-2} \varphi'(\norm{x}^{-2}).
	\end{equation}
	Indeed, let us observe that
	\begin{align*}
		u^2 \varphi'(u) 
		& = u^2 \int^\infty_0 s e^{-us} \mu(s) {\: \rm d}s \\
		& = \int^\infty_0  e^{-w} \mu(w/u) {\: \rm d} w,
	\end{align*}
	thus $u \mapsto u^2 \varphi'(u)$ is non-decreasing In particular, $\varphi'$ belongs to $\WLSC{-2}{0}{1}$.
	Hence, \eqref{eq:67} is the consequence of Proposition \ref{ubound_p_t}.
\end{remark}

\begin{lemma}\label{r^2K(r)} Suppose that there exists $\beta \in [0, d/2+1)$ such that $\varphi' \in \WLSC{-\beta}{\theta}{\underline{c}}$, 
	then there are $C_1, C_2  > 0$, depending only on $d$, such that for all $r \leq \theta^{-1/2}$
	$$C_1  \varphi'(r^{-2})\le r^2K_d(r) \le C_2\underline{c}^{-1}\left(\frac {\varphi(\theta)}  {\theta\varphi'(\theta)} +
		(d/2 - \beta + 1)^{-1}\right)
	\varphi'(r^{-2}).$$
\end{lemma}
\begin{proof} Observe that
$$(d/2  + 1)^{-1} r^{-d - 2} \varphi'(r^{-2})\le \int^{r^{-2}}_0 u^{d/2} \varphi'(u) {\: \rm d} u\le \underline{c}^{-1}\left(\frac {\varphi(\theta)}  {\theta\varphi'(\theta)} +
		(d/2 - \beta + 1)^{-1}\right)
		r^{-d - 2} \varphi'(r^{-2}).$$
		The lower bound follows from the monotonicity of $\varphi'$, while the upper bound was already proved in Lemma \ref{ubound_p_t}. Now, the conclusion follows from Lemma \ref{K_SBM}.
		\end{proof}

\begin{proposition}
	\label{prop:Glambda_Up}
	Let $\lambda > 0$. Then there is $C > 0$ such that for all $x \in \RR^d \setminus \{0\}$
	\begin{equation}
		\label{eq:73}
		G^\lambda(x) \leq C \int^{|x|^{-2}}_0 
		\frac{\varphi'(u)}{(\lambda+\varphi(u))^2}u^{d/2} 
		{\: \rm d} u.
	\end{equation}
	If the process $\mathbf{X}$ is transient, the estimate \eqref{eq:73} is also valid for $\lambda = 0$.
\end{proposition}
\begin{proof}
	Since $\phi$ is non-decreasing, by the mean value theorem, for all $r, s, t > 0$
	\begin{equation}
		\label{eq:68}
		e^{-t\varphi(s^2/r^2)}-e^{-t\varphi(4s^2/r^2)}
		\le 
		3 t e^{-t\varphi(s^2/r^2)} \varphi'(s^2/r^2) s^2 r^{-2}.
	\end{equation}
	Multiplying both sides of \eqref{eq:p_t10} by $e^{-\lambda t}$, integrating with respect to $t \in (0, \infty)$
	and finally applying the estimate \eqref{eq:68}, we obtain
	\begin{align*}
		r^{d+2} G^\lambda(r) &\leq C r^2 \int^\infty_0 t e^{-\lambda t} \int^{\infty}_{0}
		\left(e^{-t\varphi(s^2/r^2)}-e^{-t\varphi(4s^2/r^2)}\right)s^{d-1}e^{-s^2/4} {\: \rm d}s {\: \rm d}t \\
		&\leq 3C \int^\infty_0 te^{-\lambda t} \int^{\infty}_{0} e^{-t\varphi(s^2/r^2)} \varphi'(s^2/r^2) 
		s^{d+1}e^{-s^2/4} {\: \rm d}s {\: \rm d}t \\
		& = C' \int^\infty_0 \frac{\varphi'(s^2/r^2)}{(\lambda+\varphi(s^2/r^2))^2}s^{d+1}e^{-s^2/4}{\: \rm d}s.
	\end{align*}
	Since the function
	\[
		u \mapsto \frac{\varphi'(u)}{(\lambda+\varphi(u))^2}
	\]
	is non-increasing, we can estimate
	\[
		\int_{r^{-2}}^\infty \frac{\varphi'(s^2/r^2)}{(\lambda+\varphi(s^2/r^2))^2}s^{d+1}e^{-s^2/4}{\: \rm d}s
		\leq
		C
		\frac{\varphi'(r^{-2})}{(\lambda+\varphi(r^{-2}))^2}.
	\]
	Hence,
	\begin{equation}
		\label{eq:70}
		r^{d+2} G^\lambda(r) \leq C r^{d+2} \int_0^{r^{-2}}
		\frac{\varphi'(s)}{(\lambda+\varphi(s))^2} s^{d/2}e^{-s^2/4}{\: \rm d}s.
	\end{equation}
	Since $u \phi'(u) \leq \phi(u)$, and the process $\mathbf{X}$ is transient, the same argument proves \eqref{eq:73}
	for $\lambda = 0$.
\end{proof}

\begin{remark}
	Let $d\geq3$. Suppose that $\varphi$ is a special Bernstein function, that is a Bernstein function such that 
	$u \mapsto u \varphi(u)^{-1}$ is again a Bernstein function. Then there is $C > 0$ such that for all
	$x \in \RR^d \setminus\{0\}$
	\begin{equation}
		\label{eq:69}
		G(x)\leq C \norm{x}^{-d-2} \frac{\varphi'(|x|^{-2})}{\varphi^2(|x|^{-2})}.
	\end{equation}
	For the proof, we notice that $u \mapsto u^2 \varphi'(u) \varphi(u)^{-2}$ is increasing because $\varphi$ is a
	special Bernstein function (see \cite[Lemma 4.1]{KM}). Therefore the claim follows by Proposition
	\ref{prop:Glambda_Up}.

	Let us comment, that the estimate \eqref{eq:69} has already been proved in \cite{KM}, however the method used was
	different.
\end{remark}
\begin{theorem}
	\label{G_lambda}
	Let $\mathbf{X}$ be a subordinate Brownian motion on $\RR^d$ with the L\'{e}vy--Khintchine exponent 
	$\psi(x) = \varphi(|x|^2)$. 
	Let $\lambda>0$.
	Suppose that there exists $\beta \in [0, d/2+1)$ such that $(\lambda+\varphi)^{-2} \varphi' 
	\in \WLSC{-\beta}{\theta}{\underline{c}}$, then there is $C_\lambda > 0$, dependent on $\lambda, \underline{c}, \beta$ and  $\theta$   such that for all
	$\norm{x} < \theta^{-1/2}$
	\begin{equation}
		\label{eq:71}
		G^\lambda(x)
		\leq C_\lambda \frac{\varphi'(|x|^{-2})}{(\lambda+\varphi(|x|^{-2}))^2} |x|^{-d-2}.
	\end{equation}
	If the process is transient and there exists $\beta \in [0, d/2+1)$ such that $\varphi^{-2} \varphi' 
	\in \WLSC{-\beta}{\theta}{\underline{c}}$  then \eqref{eq:71} holds for  $\lambda= 0$.
	
	If additionally $(\lambda+\varphi)^{-2}\varphi' \in \WUSC{-\alpha}{\theta}{\overline{C}}$, for some
	$\alpha>0$, then there is $c_\lambda > 0$ such that for all $\norm{x} < \theta^{-1/2}$
	\[
		c_\lambda \frac{\varphi'(|x|^{-2})}{(\lambda+\varphi(|x|^{-2}))^2}|x|^{-d-2}
		\leq  G^\lambda(x).
	\]
\end{theorem}
\begin{proof}
	For $\theta \ge 0$ we set
	\[
		I(\theta)=\int^\theta_0  \frac{\varphi'(u)}{(\lambda+\varphi(u))^2}u^{d/2}	{\: \rm d}u,
		\quad
		\text{and}
		\quad
		J(\theta) = \int^\theta_0 \frac{1}{\varphi(u)}u^{d/2-1} {\: \rm d}u.
	\]
	Let  $r>\theta$. Since $\varphi'(\lambda+\varphi)^{-2} \in \WLSC{-\beta}{\theta}{\underline{c}}$, 
	\begin{align*}
		\int^{r}_0 \frac{\varphi'(u)}{(\lambda+\varphi(u))^2}u^{d/2} {\: \rm d}u
		&=I(\theta) + 
		\frac{\varphi'(r)}{(\lambda+\varphi(r))^2}\int_\theta^r 
		\frac{\varphi'(u)}{(\lambda+\varphi(u))^2}\frac{(\lambda+\varphi(r))^2}{\varphi'(r)}u^{d/2}	{\: \rm d}u\\
		&\leq I(\theta)+\frac{\varphi'(r)}{\underline{c}(\lambda+\varphi(r))^2}
		\int_\theta^r \Big(\frac{r}{u}\Big)^{-\beta}u^{d/2}	{\: \rm d}u\\
		&\leq I(\theta)+
		\frac{1}{\underline{c}(d/2+1-\beta)}\frac{\varphi'(r)}{(\lambda+\varphi(r))^2}r^{d/2+1}.
	\end{align*}
	If $\theta=0$, the proof of the upper estimate follows by Proposition \ref{prop:Glambda_Up}.

	Suppose $\theta > 0$. Since $u\varphi'(u) \le \varphi(u)$, we have
	\begin{align}
		\nonumber
		I(\theta) & \le \int^\theta_0 \frac{\varphi(u)}{(\lambda+\varphi(u))^2}u^{d/2-1} {\: \rm d}u \\
		\label{eq:72}
		& \le \int^\theta_0 \frac{1}{\varphi(u)}u^{d/2-1} {\: \rm d}u = J(\theta).
	\end{align}
	Because $(\lambda + \varphi)^{-2} \varphi'$ belongs to $\WLSC{-\beta}{\theta}{\underline{c}}$, we can
	estimate
	\begin{align}
		\nonumber
		\frac{\varphi'(r)}{(\lambda+\varphi(r))^2}r^{d/2+1}
		& \geq \underline{c} \theta^\beta r^{d/2+1-\beta}\frac{\varphi'(\theta)}{(\lambda+\varphi(\theta))^2} \\
		& \geq \underline{c} \theta^{d/2+1} \frac{\varphi'(\theta)}{(\lambda+\varphi(\theta))^2}.
		\label{WLSCb}
	\end{align}
	Now, applying \eqref{WLSCb}, we obtain
	\begin{align*}
		I(\theta) \le \lambda^{-2} \int^\theta_0 \varphi'(u) u^{d/2} {\: \rm d}u
		& \le \lambda^{-2} \varphi(\theta) \theta^{d/2} \\
		& \le \frac {\varphi(\theta)}  {\underline{c}\theta \varphi'(\theta)} 
		\frac{\varphi'(r)}{(\lambda+\varphi(r))^2}r^{d/2+1}\frac{(\lambda+\varphi(\theta))^2} {\lambda^2} .
	\end{align*}
	From the other side, by \eqref{eq:72} and \eqref{WLSCb}, we arrive at
	\[ 
		I(\theta)
		\le  J(\theta)\frac {(\lambda+\varphi(\theta))^2}  {\underline{c}\theta^{d/2+1}
		\varphi'(\theta)} \frac{\varphi'(r)}{(\lambda+\varphi(r))^2}r^{d/2+1}.
	\]
	Combining both estimates of $I(\theta)$ we obtain
	\[
		I(\theta) \le
		 \min\big\{ J(\theta) \theta^{-d/2}, \varphi(\theta)\lambda^{-2}\big\}
		\frac {(\lambda+\varphi(\theta))^2}  {\underline{c}\theta \varphi'(\theta)} 
		\frac{\varphi'(r)}{(\lambda+\varphi(r))^2}r^{d/2+1}.
	\]
	Hence, by Proposition \ref{prop:Glambda_Up}, for $\norm{x} < \theta^{-1/2}$
	\[
		G^\lambda(x) \leq C_\lambda \frac{\varphi'(|x|^{-2})}{(\lambda+\varphi(|x|^{-2}))^2}|x|^{-d-2}
	\]
	where
	\[
		C_\lambda= 
		 \min\big\{ J(\theta) \theta^{-d/2}, \varphi(\theta)\lambda^{-2}\big\}
		\frac {(\lambda+\varphi(\theta))^2}  {\underline{c}\theta \varphi'(\theta)}+ 
		\frac{1}{\underline{c}(d/2+1-\beta)}.
	\]
	In the transient case $J(\theta)$ is finite, thus the estimate \eqref{eq:71} is also valid in the case $\lambda=0$.
	
	Next, we additionally assume that $(\lambda + \varphi)^{-2} \varphi'$ belongs to
	$\WUSC{-\alpha}{\theta}{\overline{C}}$, for some $\alpha > 0$. Let $a \in (0, 1)$. Then
	\begin{align*}
		\Gamma(d/2+1) r^{d/2+1} G^\lambda(a\sqrt{r})
		&\geq 
		\int^\infty_{ar}e^{-s/r}G^\lambda(\sqrt{s})s^{d/2} {\: \rm d}s\\
		&= \int^\infty_{0}e^{-s/r}G^\lambda(\sqrt{s})s^{d/2} {\: \rm d}s
		- \int^{ar}_0e^{-s/r}G^\lambda(\sqrt{s})s^{d/2} {\: \rm d}s.
	\end{align*}
	Let us observe that if $f(s) = G^\lambda(B_{\sqrt{s}})$ then analogous calculation to \eqref{eq:27} shows that
	\begin{equation}
		\label{eq:127}
		\gamma \calL f(\gamma) 
		=\frac{2^{1-d}}{\Gamma(d/2)} \int_0^{\infty} e^{-r^2/4} r^{d-1} \frac{{\rm d}r}{\lambda+\psi(r \sqrt{\gamma})}.
	\end{equation}
	for any $\gamma > 0$. Since $G^\lambda$ is non-increasing we have
	\[
		G^\lambda(\sqrt{s})s^{d/2} \ge C G^\lambda\big(B_{2\sqrt{s}}\setminus B_{\sqrt{s}}\big)
	\]
	hence
	\begin{align*}
		r^{-1} \int^\infty_{0} e^{-s/r} G^\lambda(\sqrt{s})s^{d/2} {\: \rm d}s
		&\geq 
		C r^{-1} \int^\infty_{0} e^{-s/r}G^\lambda\big(B_{2\sqrt{s}}\setminus B_{\sqrt{s}}\big) {\: \rm d}s\\
		&=
		C \int^\infty_0 \left(\frac{1}{\lambda+\varphi(s^2/(4r))}-\frac{1}{\lambda+\varphi(s^2/r)}\right)
		e^{-s^2/4}s^{d-1} {\: \rm d}s.
	\end{align*}
	By the mean value theorem and monotonicity of $\varphi$ and $\varphi'$,
	\begin{align*}
		r^{-1} \int_0^\infty e^{-s/r} G^\lambda(\sqrt{s})s^{d/2} {\: \rm d}s
		&\geq
		C \int^\infty_0 \frac{3s^2}{r}\frac{\varphi'(s^2/r)}{(\lambda+\varphi(s^2/r))^2}e^{-s^2/4}s^{d-1} {\: \rm d}s\\
		& \geq 
		C r^{-1} \frac{\varphi'(1/r)}{(\lambda+\varphi(1/r))^2} \int^{1/2}_0 e^{-s^2/4}s^{d+1} {\: \rm d}s\\
		& =
		C' r^{-1} \frac{\varphi'(1/r)}{(\lambda+\varphi(1/r))^2}. 
	\end{align*}
	Since $(\lambda + \varphi)^{-2} \varphi'$ belongs to $\WUSC{-\alpha}{\theta}{\overline{C}}$, by
	\eqref{eq:71}, for $r \leq \theta^{-1}$ we get
	\begin{align*}
		\int^{ar}_0 e^{-s/r} G^\lambda(\sqrt{s})s^{d/2}	{\: \rm d}s 
		& \leq C_\lambda \int^{ar}_0 \frac{\varphi'(s^{-1})}{(\lambda+\varphi(s^{-1}))^2} \frac{{\rm d}s}{s}\\
		& \leq C_\lambda \overline{C} \frac{\varphi'(r^{-1})}{(\lambda+\varphi(r^{-1}))^2}
		\int^{ar}_{0}(s/r)^{\alpha}\frac{{\rm d}s}{s} \\
		& = \frac{a^\alpha} \alpha C_\lambda \overline{C} \frac{\varphi'(r^{-1})}{(\lambda+\varphi(r^{-1}))^2}.
	\end{align*}
	Hence, we obtain
	\[
		\Gamma(d/2+1) r^{d/2+1} G^\lambda(a\sqrt{r})
		\geq
		\Big(C'-\frac{a^\alpha}\alpha C_\lambda \overline{C}\Big)
		\frac{\varphi'(r^{-1})}{(\lambda+\varphi(r^{-1}))^2}.
	\]
	Now, setting
	\[
		a = \left(\frac{C' \alpha }{2C_\lambda \overline{C}} \right)^{1/\alpha},
	\]
	we obtain
	\[
		G^\lambda(ar) \geq\frac{ C'}{\Gamma(d/2+1)} \frac{\varphi'(r^{-2})}{(\lambda+\varphi(r^{-2}))^2}r^{-d-2}.
	\]
	Finally, for $\norm{x} \leq a \theta^{-1/2}$, by monotonicity of $\varphi$ and $\varphi'$ we have
	\[
		G^\lambda(x) \geq  \frac{ C'}{\Gamma(d/2+1)} a^{d+2}\frac{\varphi'(\norm{x}^{-2})}{(\lambda+\varphi(\norm{x}^{-2}))^2}\norm{x}^{-d-2}.
	\]
	To get the lower estimate for  $\norm{x}\in \big(a \theta^{-1/2}, \theta^{-1/2}\big)$ we may use Proposition
	\ref{HKLB1} and the fact that the density of the L\'evy measure is everywhere positive.
\end{proof}
A version of Theorem \ref{G_lambda} for $\lambda=0$ was proved in \cite[Proposition 4.5]{KM} where in addition to scaling
properties of $\varphi'$ (slightly different than ours) it was also assumed that  the potential density a of the subordinator was decreasing on $(0, \infty)$.

\begin{theorem}
	\label{p_t_bound}
	Let $\mathbf{X}$ be a subordinate Brownian motion on $\RR^d$ with the symbol $\varphi(|x|^2)$.
	Suppose that there exists $\beta \in (0, d/2+1)$ and  $\alpha > 0$ such that 
	$\varphi'\in \WLSC{-\beta}{\theta}{\underline{c}} \cap  \WUSC{-\alpha}{\theta}{\overline{c}}$, then
	for all $t > 0$ and and $x \in \RR^d$, if $t \varphi(1/ \norm{x}) \leq 1$ and $\norm{x} \leq \theta^{-1/2}$
	\[
		p(t, x) \asymp t \norm{x}^{-d-2} \varphi'(|x|^{-2}).
	\]
\end{theorem}
\begin{proof}
	In view of Proposition \ref{ubound_p_t}, it is enough to show the lower estimate. First, we find a lower bound of the L\'evy measure. Let $a<1$ and $r<\theta^{-1/2}$, then by monotonicity of the density of the L\'evy measure
	$$r^2K_d(r)- (ra)^2K_d(ar)=  \omega_d\int_{ar}^ru^{d+1}\nu(u)du\le \omega_d\nu(ar)r^{d+2}. $$
	Next, we find $a$, small enough, such that  $\frac {r^2K_d(r)} {(ra)^2K_d(ar)}\ge 2$. Indeed,  by Lemma \ref{r^2K(r)}, $u^2K(u) \asymp \varphi'(u^{-2}), u\le \theta^{-1/2} $, hence the application of  
	$\varphi'\in \WUSC{-\alpha}{\theta}{\overline{c}}$ provides the desired $a$. Therefore 
	$$\omega_d\nu(ar)r^{d+2}\ge  (ra)^2K_d(ar) \asymp \varphi'((ar)^{-2})$$
	This implies that 
	\[
		\nu(x) \ge C \frac{\varphi'(|x|^{-2})}{|x|^{d+2}}, \quad \norm{x} \leq \theta^{-1/2} 
	\]
	and the conclusion is a consequence of Proposition \ref{HKLB1}.
\end{proof}

\begin{proposition}
	\label{nuapprox}
	Let $d\ge 3$. Suppose the density of the L\'evy measure $\mu$ of the subordinator is non-increasing. Then there is
	$c > 0$ such that for all $x \in \RR^d \setminus \{0\}$
	\begin{equation}
		\label{eq:74}
		\nu(x) \leq c \norm{x}^{-d-4} |\varphi''(|x|^{-2})|.
	\end{equation}
\end{proposition}
\begin{proof}
	First, we observe that for $t > 0$
	\begin{equation}
		\label{mu_estimate}
		\mu(t) \leq 3et^{-3} |\varphi''(t^{-1})|.
	\end{equation}
	Since $\mathcal{L}(s\mu)(x)=\varphi'(x)$ and $\mu$ is non-increasing, the estimate \eqref{eq:74} follows by
	\cite[Lemma 5]{MR3165234} applied to $f(t)=\mu(t)$ and $m=n=1$.

	By \eqref{eq:75}, we have
	\[
		\nu(x) < \int_0^{\norm{x}^2} (4\pi u)^{-d/2} e^{-\norm{x}^2/(4s)} \mu(u) {\: \rm d}u + 
		\int_{\norm{x}^2}^\infty (4\pi u)^{-d/2} \mu(u) {\: \rm d}u.
	\]
	By \eqref{mu_estimate} and the monotonicity of $|\varphi''|$,
	\begin{align*}
		\int_0^{\norm{x}^2} (4\pi u)^{-d/2} e^{-\norm{x}^2/(4 s)} \mu(u) {\: \rm d}u
		& \leq
		3 e \int_0^{\norm{x}^2} (4 \pi u)^{-d/2} e^{-\norm{x}^2/(4s)} \frac{\abs{\varphi''(u^{-1})}}{u^3} {\: \rm d}u \\
		& \leq
		C \frac{\abs{\varphi''(\norm{x}^{-2})}}{\norm{x}^{d+4}}.
	\end{align*}
	Similarly, by \eqref{mu_estimate} and monotonicity of $\mu$
	\begin{align*}
		\int_{\norm{x}^2}^\infty (4 \pi u)^{-d/2} \mu(u) {\: \rm d} u
		& \leq
		\mu(\norm{u}^2) \int_{\norm{x}^2}^\infty u^{-d/2} {\: \rm d} u\\
		& \leq C \frac{\abs{\varphi''(\norm{x}^{-2})}}{\norm{x}^{d+4}},
	\end{align*}
	which ends the proof.
\end{proof}

\begin{theorem}
	\label{thm:9}
	Let $\mathbf{X}$ be a subordinate Brownian motion on $\RR^d$ with the L\'{e}vy--Khintchine
	exponent $\psi(x) = \varphi(|x|^2)$. Suppose that the density of the L\'evy measure $\mu$ of the subordinator
	is non-increasing and $d \ge 3$. Then the following are equivalent:
	\begin{enumerate}
		\item 
		\label{thm:9:1}
		There exist  $\theta \ge 0, \alpha>0$ such that $\varphi'\in  \WUSC{-\alpha}{\theta}{\overline{c}}$.
		\item 
		\label{thm:9:2}
		There are $C >0$ and $\theta>0$ such that for all $t > 0$ and $x \in \RR^d$, if
		$t \varphi(\norm{x}^{-2}) \leq 1$ and $\norm{x} < \theta^{-1}$ then
	\[
		p(t, x) \ge C t \norm{x}^{-d-2} \varphi'(|x|^{-2}).
	\]
		\item 
		\label{thm:9:3}
		There are $C > 0$ and $\theta>0$ such that if $\norm{x} < \theta^{-1}$ then
	\[
		\nu( x)\ge C \norm{x}^{-d-2} \varphi'(|x|^{-2}).
	\]
	\end{enumerate}
\end{theorem}
\begin{proof}
	Since $u \mapsto u^2 \varphi'(u)$ is non-decreasing, $\varphi' \in \WLSC{-2}{0}{1}$. Therefore,
	(i)~$\Rightarrow$~(ii) is the consequence of Theorem \ref{p_t_bound}. The implication (ii)~$\Rightarrow$~(iii)
	follows because
	\[
		\lim_{t \to 0^+} t^{-1} p(t, x) = \nu(x)
	\]
	vaguely on $\RR^d \setminus \{0\}$, see e.g. the proof of \cite[Theorem 6]{cgt}. It remains to prove
	that (iii) implies (i).
	
	First, let us observe that by Proposition \ref{nuapprox} there is a constant $c > 0$ such that for $r > 0$
	\[
		\nu(r)\le - c \frac{\varphi''(r^{-2})}{r^{d+4}}.
	\]
	Therefore, for $\lambda \geq \sqrt{\theta}$ 
	\[
		C c^{-1} \varphi'(\lambda) \leq -  \lambda \varphi''(\lambda).
	\]
	Hence, $\lambda \mapsto \lambda^{C/c} \varphi'(\lambda)$ is non-increasing for $\lambda \ge \sqrt{\theta}$,
	thus $\varphi' \in \WUSC{-C/c}{\sqrt{\theta}}{1}$ and the proof is completed. 
\end{proof}

\section{Examples}
\label{sec:6}

\begin{example}
Let $\nu(x)=\ind{B(0,1)}|x|^{-d}.$ Then, for $|x|>1$,
\begin{align*}
	\psi(x)-\psi(1)
	& = \int_{B_{|x|} \setminus B_1} (1-\cos ( u_1)) \nu(u)	{\: \rm d}u \\
	& = \omega_d \ln(|x|)- 2\int^{|x|}_{1}\cos (u) \nu_1(u) {\: \rm d}u,
\end{align*}
thus $\psi \in \Pi^\infty_{\omega_d}$. Moreover, for all $t \in (0, 1)$ and $\norm{x} \leq 1$
\[
	e^{-t\psi(1/|x|)} \asymp |x|^{2\omega_d t},
\]
Hence, by Theorem \ref{thm:4} there are $t_0 > 0$ and $r_0>0$ such that for all $t \in (0, t_0)$ and $x \in B_{r_0}$
\[
	p(t, x) \asymp  t|x|^{ 2\omega_d t-d}.
\]
By Theorem \ref{GUB3} and Proposition \ref{HKLB1}, we have $r_0 = 1$.
\end{example}

\begin{example}
	\label{exmp:2}
	Let $\mathbf{S} = (S_t : t \geq 0)$ be L\'evy process with the L\'evy--Khintchine exponent $\norm{x}^\alpha$
	for $\alpha \in (0, 2]$ and $(T_t : t \geq 0)$ a subordinator with the Laplace exponent
	$\phi(\lambda) = \big(\log(1+\lambda)\big)^\beta$ for $\beta \in (0, 1]$. Then $\mathbf{X} = (X_t : t\geq 0)$
	where $X_t = S_{T_t}$ has the L\'evy--Khintchine exponent
	\[
		\psi(\lambda) = \big(\log(1 + \abs{\lambda}^\alpha)\big)^{\beta}.
	\]
	Thus $\psi \in \Pi^\infty_\ell$ for
	\[
		\ell(\lambda) = \alpha^\beta \beta \big(\log(1+\abs{\lambda}^\alpha)\big)^{\beta - 1}.
	\]
	By Theorem \ref{thm:4}, there are $t_0 > 0$ and $r_0 > 0$ such that for all $t \in (0, t_0)$ and $\norm{x} \leq r_0$
	\[
		p(t, x) \asymp t \norm{x}^{-d} \big(\log(1 + \norm{x}^{-1})\big)^{\beta - 1}
		\exp\Big(-t \big(\log(1+\norm{x}^{-\alpha})\big)^{\beta}\Big).
	\]
\end{example}	
To present the  complicated nature of the estimates   of the heat kernels from Example \ref{exmp:2},
we provide global estimates in the cases: $\alpha \in (0,2)$ and  $\beta = 1$ or $\beta = 1/2$.
\begin{example}
	Let $\mathbf{X}$ be a L\'evy process with the L\'evy--Khintchine exponent
	\[
		\psi(\lambda) = \log (1 + \abs{\lambda}^\alpha).
	\]
	Let $s(t, x)$ be the heat kernel for the standard isotropic $\alpha$-stable L\'evy process. Then we have
	\[
		p(t, x)=\frac{1}{\Gamma(t)} \int^\infty_0 e^{-u} u^{t-1} s(u,x) {\: \rm d}u.
	\]
	Recall that
	\[
		s(t,x) \asymp \min\{t^{-d/\alpha},t|x|^{-d-\alpha}\}.
	\]
	First, let us note that $K_d(x) \asymp \min\{1, |x|\}^{-\alpha}$ and $\nu(x) \asymp \min\{|x|^{-d}, |x|^{-d-\alpha}$.
	Hence, by \cite[Theorem 1 and Remark 2]{cgt}, Theorem \ref{GUB3} and Proposition \ref{HKLB1}, there is
	$t_1 > d/\alpha$ such that for all $t \geq t_1$ and $x \in \RR^d$
	\[
		p(t,x) \asymp \min\big\{t^{-d/\alpha},t|x|^{-d-\alpha}\big\} \asymp s(t,x),
	\]

	Let
	\[
		\Gamma(t,x)=\int^\infty_xe^{-u}u^{t-1} {\: \rm d}u
		\quad\text{and}\quad
		\gamma(t,x)=\int^x_0e^{-u}u^{t-1} {\: \rm d}u
	\]
	be two incomplete Gamma functions. Then
	\[
		p(t, x)
		\asymp
		\frac{1}{\Gamma(t)}
		\big(|x|^{-d-\alpha} \gamma(t+1,|x|^\alpha)+\Gamma(t-d/\alpha,|x|^\alpha)\big).
	\]
	For $t \in (0, t_1)$, we have
	\[
		\Gamma(t)\asymp t,
		\quad\text{and}\quad
		\gamma(t+1,|x|^\alpha)\asymp 
		\min\big\{1, |x|^{\alpha(t+1)}\big\}.
	\]
	Moreover, for $|x| \geq 1$
	\[
		\Gamma(t-d/\alpha,|x|^\alpha)
		\leq e^{-|x|^\alpha/2}\int^\infty_1e^{-u/2}u^{t_1-d/\alpha-1}
		{\: \rm d}u
		\leq C \norm{x}^{-d-\alpha}.
	\]
	Whereas, for $|x|\leq 1$ we have
	\begin{align*}
		\Gamma(t-d/\alpha,|x|^\alpha) 
		& \asymp 
		\begin{cases}
			\log 2 |x|^{-\alpha}, &\text{for } t = d/\alpha, \\
			\frac{2^{t-d/\alpha}-|x|^{\alpha t-d}}{t-d/\alpha}, &\text{for } t \neq d/\alpha,
		\end{cases} \\
		& \asymp 
		\min\bigg\{ \log 2 |x|^{-\alpha}, \frac{\max\{1, |x|\}^{\alpha t-d}}{|t-d/\alpha|}\bigg\}.
	\end{align*}
	Hence, for $t<t_1$,
	\[
		p(t, x) \asymp
		\begin{cases}
			t|x|^{-d-\alpha}, & \text{for } |x|\geq 1, \\
			t\min\big\{\log 2 \norm{x}^{-\alpha}, (t-d/\alpha)^{-1} \big\},
			& \text{for } \norm{x} < 1 \text{ and } t > d / \alpha, \\
            t (\log 2|x|^{-\alpha} + |x|^{\alpha t-d}), & \text{for } |x|<1 \text{ and } t \in (0, d/\alpha].
		\end{cases}
	\]
	Let us notice, that if $t\psi(1/|x|)\leq 1$ then 
	\[
		p(t, x)\asymp t\nu(x),
	\]
	and
	\[
		p(t,0) \asymp 
		\begin{cases}
			\frac{t^{1-d/\alpha}}{t-d/\alpha} & \text{for } t > d/\alpha, \\
			\infty & \text{for } t \in (0, d/\alpha].
		\end{cases}
	\]
\end{example}						
					
In order to study the next example we need the following lemma which easily follows from \cite[Section 6]{MR1390191}.
\begin{lemma}
	\label{K_estimate}
	Let
	\[
		H(x, a)= \int_0^a v^{-1/2}e^{-v} e^{-\frac{t^2a}{4v}} {\: \rm d}v.
	\]
	Then
	\[
		H(t, a) \asymp
		\begin{cases}
			\sqrt{\pi} e^{-t\sqrt{a}} & \text{for } t \in (0, 1 + 2\sqrt{a}], \\
			\frac{\sqrt{a}}{t^2 - 4a} \sqrt{\pi} e^{-t^2/4 - a^2}.
		\end{cases}
	\]
\end{lemma}

\begin{example}
	Let $\mathbf{X}$ be the L\'evy process with the L\'evy--Khintchine exponent
	\[
		\psi(\lambda) = \big(\log (1 + \abs{\lambda}^\alpha)\big)^{1/2}.
	\]
	By $p_\alpha(t, x)$ we denote the transition density of the geometric $\alpha$-stable process and by
	$s(t, x)$ we denote the transition density of the standard isotropic $\alpha$-stable process.
	If $(T_t : t \geq 0)$ is an independent of $1/2$-stable subordinator, then
	\[
		p(t, x) = \EE p_\alpha(T_t, x).
	\] 
	For $|x|\ge 1$ and $t>0$, by Example 3, we have $p_\alpha(T_t,x)\asymp s(T_t, x)$, hence $p(t, x)$ is comparable
	to $\alpha/2$-stable transition density, that is for all $t > 0$ and $\norm{x} \geq 1$
	\begin{align*}
		p(t, x) & \asymp \EE s(T_t,x) \\
		& \asymp \min\big\{ t |x|^{-d-\alpha/2}, t^{-2d/\alpha}\big\}.
	\end{align*}
	Let $t_1$ be chosen in Example 3. In fact we may take $t_1= 3d/\alpha$. 

	For  $|x|<1$ we write 
	\begin{align*}
		p(t, x) & =  \EE\big(p^\alpha(T_t,x), 0 < T_t < d/\alpha\big)
		+\EE\big(p^\alpha(T_t,x), d/\alpha \le T_t < 3d/\alpha\big)
		+\EE\big(p^\alpha(T_t,x), T_t\ge 3d/\alpha\big) \\
		& =I_1 + I_2 + I_3.
	\end{align*}
	By Example 3, we have
	\[
		I_3
		\asymp 
		t^{-2d/\alpha}\min\big\{(1, t^{1+2d/\alpha}\big\}.
	\]
	Again by Example 3, we get
	\begin{align*}
		I_2
		& \asymp
		t \int_{d/\alpha}^{3d/\alpha}
		v \min\big\{\log 2 |x|^{-\alpha}, (v-d/\alpha)^{-1} \big\} 
		v^{ -3/2} e^{-t^2/4v} {\: \rm d} v \\
		& \asymp 
		t\int_{0}^{2d/\alpha} \min\big\{ \log 2 |x|^{-\alpha}, v^{-1} \big\} e^{-t^2/4v} {\: \rm d} v,
	\end{align*}
	and 
	\[
		I_1 \asymp
		t \int_{0}^{d/\alpha} v
		\Big(\log 2 |x|^{-\alpha} + |x|^{\alpha v-d}\Big) v^{-3/2} e^{-t^2/4v} {\: \rm d}v
		= I_{11}+I_{12}.
	\]
	Observe that for $t^2 \leq d \alpha^{-1}$, we have
	\[
		I_{11} = t \log 2 |x|^{-\alpha} \int_{0}^{d/\alpha}  
		v^{-1/2} e^{-t^2/4v} {\: \rm d} v
		\asymp  
		t \log 2 |x|^{-\alpha}.
	\]
	and
	\begin{align*}
		I_{12} = & t \norm{x}^{-d} \int_{0}^{d/\alpha} |x|^{\alpha v} v^{-1/2} e^{-t^2/4v} {\: \rm d}v\\
		&= t |x|^{-d} \int_{0}^{d/\alpha} e^{ \alpha v \log |x|} v^{-1/2}e^{-t^2/4v} {\: \rm d}v\\
		&= t |x|^{-d} (-\alpha \log |x|)^{-1/2} 
		\int_{0}^{-d \log |x|} e^{-v} v^{-1/2} e^{t^2\alpha\log |x|/4v} {\: \rm d} v.
	\end{align*}
	In view of Lemma \ref{K_estimate}, for $t \sqrt{\alpha/d} \le 1 + 2 (- d \log |x|)^{1/2}$ and  $|x|<1$, 
		we have 
	\[
		I_{12} \asymp t \norm{x}^{-d} ( -d \log 2 |x|^{-\alpha})^{-1/2} 
		e^{-t \sqrt{-\alpha \log |x|}}.
	\]
	If $t \sqrt {\alpha/d} \ge 1+2 \sqrt{- d\log |x|}$ then $I_1+I_2 \le C e^{-ct^2}$, for some $c, C>0$ hence
	\[
		p(t, x) \asymp I_3 \asymp t^{-2d/\alpha}.
	\]
	Combining all the estimates, we conclude that for $t \sqrt {\alpha/d} \le 1+2 \sqrt{- d\log|x|}$ and
	$|x|\le 1/2$ we obtain
	\[
		p(t, x) \asymp I_1 
		\asymp 
		t \norm{x}^{-d} (-\alpha \log |x|)^{-1/2} 
		e^{-t \sqrt{-\alpha \log |x|}} - t \alpha \log |x|.
	\]
	If $t\sqrt {\alpha/d} \ge 1+2 \sqrt{-d\log |x|}$ and $|x|\le 1/2$ then
	\[
		p(t, x) \asymp t^{-2d/\alpha},
	\]
	and finally, if $|x| \ge 1/2$ and $t > 0$ we have
	\[
		p(t, x) \asymp \min\big\{ t \norm{x}^{-d-\alpha/2}, t^{-2d /\alpha}\big\}.
	\]
	Let us observe that for $t\sqrt{\log (1+ |x|^{-\alpha})} \le 1$ we have
	\[
		p(t, x) \asymp t \nu(x)
	\]
	for all $x \in \RR^d$, and
	\[
		p(t, 0) = \infty
	\]
	for all $t > 0$.
\end{example}

\begin{bibliography}{unimodal}
	\bibliographystyle{amsplain}
\end{bibliography}
\end{document}